\documentclass{article}

\usepackage{amscd,amssymb,amsthm}
\usepackage{graphicx}
\usepackage[dvipsnames]{xcolor}
\usepackage{cancel}
\usepackage{mathrsfs}
\usepackage[centertags]{amsmath}
\usepackage{amsfonts}
\usepackage{amssymb}
\usepackage{amsthm}
\usepackage{amsmath}
\usepackage{newlfont}
\usepackage{pifont}
\usepackage{fancyhdr}
\usepackage{pdfpages}

\newtheorem{theorem}{Theorem}
\newtheorem{lemma}{Lemma}

\newtheorem{corollary}{Corollary}

\newtheorem*{theorem*}{Theorem}
\newtheorem{proposition}{Proposition}
\newtheorem{remark}{Remark}

\usepackage[utf8]{inputenc}
\begin{document}

\begin{center}
\large{Absolutely Continuous Spectrum for Parabolic Flows/Maps}
\end{center}

\begin{center}
Lucia D. Simonelli
\end{center}
\begin{center}
\textit{Abdus Salam International Centre for Theoretical Physics (ICTP) - Trieste, Italy}
\end{center}
\begin{center}
lucia.simonelli@ictp.it
\end{center}


%

\bigskip


\begin{abstract}
We provide an abstract framework for the study of certain spectral properties of parabolic systems; specifically, we  determine under which general conditions to expect the presence of absolutely continuous spectral measures. We use these general conditions to derive results for spectral properties of time-changes of unipotent flows on homogeneous spaces of semisimple groups regarding absolutely continuous spectrum as well as maximal spectral type; the time-changes of the horocycle flow are special cases of this general category of flows. In addition we use the general conditions to derive spectral results for twisted horocycle flows and to rederive certain spectral results for skew products over translations and Furstenberg transformations.
\end{abstract}

\section{Introduction}

\subsection{Motivation}

\noindent Spectral theory of dynamical systems has long been studied \cite{lema}; of particular interest, is the notion of when to expect the presence of absolutely continuous spectral measures. Since absolutely continuous spectrum implies mixing, this property can be thought of as an indicator of how chaotic, or how far from orderly, a system is. In the hyperbolic setting, systems are characterized as having a correlation decay that is exponential. As a result, techniques derived from the existence of a spectral gap as well as probabilistic tools are available for the study of spectral properties, and therefore, it is in the hyperbolic setting where the existence of absolutely continuous spectrum predominantly occurs. Interestingly, certain parabolic systems also share this property despite having at most polynomial decay of correlations. This slower decay of correlations precludes the use of the tools available in the spectral study of hyperbolic systems, and consequently, spectral theory of smooth parabolic flows and smooth perturbations of well known parabolic flows has been much less studied. This work is devoted to creating an abstract framework for the study of certain spectral properties of parabolic systems. Specifically, we attempt to answer the question: under what general conditions can we expect the existence of absolutely continuous spectral measures?

\subsection{Statement of results}

\noindent In Theorem \ref{main} in we present general conditions under which we expect a skew-adjoint operator to have absolutely continuous spectrum. The proof of this theorem is a general application of the method in \cite{forniulci},\break\newpage

 \noindent in which the authors show that the Fourier Transform of the spectral measures of smooth coboundaries are square integrable. The method in \cite{forniulci} was inspired by Marcus's proof of mixing of horocycle flows \cite{marcusfigure} which requires a specific form of tangent dynamics from which one can exploit shear of nearby trajectories. In addition, we rely on the bootstrap technique from \cite{forniulci} to estimate the decay of correlations of specific smooth coboundaries. Our choice of coboundaries depends upon a growth condition involving the commutator of the skew-adjoint operator and a certain auxiliary operator.
%

We use this general, functional analytic result to derive the following results for certain parabolic dynamical systems. Theorem \ref{unipotent} states that time-changes of unipotent flows on homogeneous spaces of semisimple groups have absolutely continuous spectrum. In the compact case, we also show that the maximal spectral type is Lebesgue, following the method in \cite{forniulci}. The time-changes of the horocycle flow are special cases of this general category of flows, and spectral properties of the time-changes of horocycle flows were shown in \cite{forniulci}, \cite{tiedrahor}, and \cite{tiedra}. In addition, we use the general conditions to prove Theorem \ref{twisted} and Corollary \ref{moretwisted} regarding spectral results for twisted horocycle flows, combining the horocycle time-change with a circle rotation.  Lastly, we rederive certain spectral results for skew products over translations and Furstenberg transformations, originally shown in \cite{tiedra}. Our results are slightly weaker as we prove results for functions of class $C^{2}$ while the author in \cite{tiedra} considers functions of class $C^{1}$ with an added Dini Condition.


\begin{remark} It is standard to consider spectral decomposition in the setting of self-adjoint operators. When we consider flows that are represented as strongly  continuous one-parameter unitary groups, the generating operators (vector fields) are essentially skew-adjoint. Since multiplication by $i$ gives an essentially self-adjoint operator we make no distinction from the standard setting, and thus, directly apply the theory for self-adjoint operators.
\end{remark}

\vspace*{-8pt}
\section{Abstract conditions}


\subsection{Preliminary assumptions}
\noindent Suppose that a closed operator $X$ on a Hilbert Space $\mathcal{H}$, defined on a dense subspace $D$ such that $X(D) \subset D$, generates a strongly continuous, one parameter group $\{e^{sX}\}$.

Suppose also that $e^{tU}(D)\subset D$ is a strongly continuous, unitary group with infinitesimal generator $U$, and that the commutator
\[
H(t)=e^{-tU}\left [X,e^{tU} \right ]
\]
\noindent is defined on $D$.

For $u \in D$, let
\[
\frac{H(t)}{t^{\beta}}u
\overset{\mathcal{H}}{\underset{t \to \infty}\longrightarrow} Hu,
\]
\noindent such that $H(D) \subset D$ and $\overline{Ran(H)}=\overline{\left \{Hu:u \in D \right \}}=\mathcal{H}$.


For $B_{1}$, $B_{2}$ bounded operators on $\mathcal{H}$ such that 
$B_{2}: D \to D$, let
\[
\left \Vert \; \langle e^{tU}f,f \rangle_{\mathcal{H}}\; \right \Vert_{L^{2}(\mathbb{R})} \; \leq\; \left  \Vert  \frac{1}{\sigma}\int_{0}^{\sigma} \langle e^{sX}e^{tU}f, B_{1}e^{sX}B_{2}f \rangle _{\mathcal{H}} ds\> \right \Vert_{L^{2}(\mathbb{R})}.
\]

\begin{theorem} \label{main}

\noindent If for $\beta > \frac{1}{2}$, $H(t)$ and $H$ satisfy:

\begin{itemize}

\item[(i)] $\frac{H(t)}{t^{\beta}}H^{-1}$ is defined on $Ran(H)$, extends by continuity to a bounded operator\break\newpage

 \noindent $\frac{\tilde{H}(t)H^{-1}}{t^{\beta}}$ on $\mathcal{H}$ with uniformly bounded  $\left \Vert \cdot \right \Vert_{op}$ norm in $t$, and satisfies
\[
\limsup_{t \to \infty} \left \Vert I-\frac{\tilde{H}(t)H^{-1}}{t^{\beta}} \right \Vert_{op} < 1.
\]

\item[(ii)] $ \left [X,\frac{H(t)}{t^{\beta}}H^{-1} \right ]$ is defined on $Ran(H)$ and extends by continuity to a bounded operator on $\mathcal{H}$ with uniformly bounded  $\left \Vert \cdot \right \Vert_{op}$ norm in $t$

\item[(iii)] $\big [H(t),H \big ]H^{-1}$ is defined on $Ran(H)$ and extends by continuity to a bounded operator on $\mathcal{H}$ with uniformly bounded  $\left \Vert \cdot \right \Vert_{op}$ norm in $t$

\end{itemize}
\noindent then for all $f \in \overline{Ran(H)}=\mathcal{H}$, the associated spectral measures,  $\mu_{f}$, of $U$ are absolutely continuous.

 \end{theorem}

\begin{remark} Often in ergodic theory, $\mathcal{H}$ is a subspace of a larger Hilbert Space; for example, $\mathcal{H}=L^{2}_{0}(M)$ the space of zero-average functions in $L^{2}(M)$. While in this setting the Theorem doesn't give a result for purely absolutely continuous spectrum it implies the existence of an absolutely continuous component.
\end{remark}

\begin{remark} The utilization of the operator $H(t)$ is suggested by the method in \cite{forniulci} which was based on Marcus's shear mechanism in \cite{marcusfigure}. The authors in \cite{tiedraETDS} and \cite{richard} have used a similar term to derive criteria for strong mixing.
\end{remark}

%
%


\begin{proof}
\noindent Let $f \in Ran(H)$ and let $\hat{\mu}_{f}(t)=\int_{\mathbb{R}} e^{it \xi} d\mu_{f}(\xi)$ be the Fourier Transform of the spectral measure $\mu_{f}$.
\[
\left \Vert  \hat{\mu}_{f}(t) \right \Vert_{L^{2}(\mathbb{R})} =\left \Vert \; \langle e^{tU}f,f \rangle _{\mathcal{H}} \> \right \Vert_{L^{2}(\mathbb{R})}\; \leq\;\left  \Vert \frac{1}{\sigma}\int_{0}^{\sigma} \langle e^{sX}e^{tU}f, B_{1}e^{sX}B_{2}f \rangle_{\mathcal{H}} \> ds \> \right \Vert_{L^{2}(\mathbb{R})}.
\]
\noindent For $s \in [0, \sigma]$, we integrate by parts:
\[
\hspace{2cm} B_{1}e^{sX}B_{2}f \hspace{2cm}  \frac{d}{ds}\left (B_{1}e^{sX}B_{2}f \right ) = B_{1}e^{sX}X(B_{2}f)
\]
\[
e^{sX}e^{tU}f \hspace{1cm}  \hspace{2cm} \int_{0}^{\sigma} e^{sX}e^{tU}fds.
\]
\[
\begin{split}
\frac{1}{\sigma}\int_{0}^{\sigma} \langle e^{sX}e^{tU}f, B_{1}e^{sX}B_{2}f \rangle _{\mathcal{H}} \>ds
=& \> \frac{1}{\sigma} \langle \int_{0}^{\sigma} e^{sX}e^{tU}fds, B_{1}e^{\sigma X}B_{2}f \rangle _{\mathcal{H}}\\
 -& \frac{1}{\sigma}\int_{0}^{\sigma} \langle \int_{0}^{S}e^{sX}e^{tU}fds, B_{1}e^{sX}X(B_{2}f) \rangle _{\mathcal{H}} dS
\end{split}
\]
\noindent
From our assumptions, both $B_{1}e^{sX}B_{2}f$ and $ B_{1}e^{sX}X(B_{2}f)$ are bounded in $\mathcal{H}$. Thus, in order to show that $\hat\mu_{f}(t) = O(\frac{1}{t^{\beta}})$, we need a bound (in $t$) for
\[
\left \Vert \int_{0}^{\sigma}e^{sX}e^{tU}f\>ds \right \Vert_{\mathcal{H}}.
\]

%

Suppose that conditions $(i)$, $(ii)$, and $(iii)$ hold, and let $f$ be  a coboundary of the form $f=Hg$, for $g \in Dom(H)$:
\[
\begin{split}
\int_{0}^{\sigma}e^{sX}e^{tU}f\>ds =& \int_{0}^{\sigma}e^{sX}e^{tU}Hg\>ds \\
=&\underbrace{\int_{0}^{\sigma}e^{sX}e^{tU}\left (H-\frac{H(t)}{t^{\beta}} \right )g\>ds}_\text{I.}+\underbrace{\int_{0}^{\sigma}e^{sX}e^{tU}\frac{H(t)}{t^{\beta}}g\> ds}_\text{II.}.
\end{split}
\]

\noindent I.  \hspace{0.5cm} Let
\[
\tilde{H}(s,t)=e^{sX}e^{tU}\left (I-\frac{\tilde{H}(t)H^{-1}}{t^{\beta}}\right )e^{-tU}e^{-sX},
\]

\noindent where $\frac{\tilde{H}(t)H^{-1}}{t^{\beta}}$ is the bounded extension of $\frac{H(t)H^{-1}}{t^{\beta}}$.

%
%
%
%
Since $\{e^{sX}\}$ is strongly continuous, for $f \in \mathcal{D}$,
\[
s\!\!-\!\!\lim_{s \to 0} e^{sX}f = If.
\]
\noindent Thus,
\[
\sup_{s \in [0,\sigma]} \Vert e^{sX} \Vert_{op} \leq 1 + k(\sigma)
\]
\noindent where
\[
\lim_{\sigma \to 0} k(\sigma) = 0.
\]
\noindent Since $\limsup\limits_{t \to \infty} \left \Vert I-\frac{\tilde{H(t)}H^{-1}}{t^{\beta}} \right \Vert_{op} < 1$, then for large $t$ and small enough $\sigma$,
\[
 \Vert \tilde{H}(s,t) \Vert _{\mathcal{H}} < C_{1} < 1.
\]
\noindent We compute, for $f=e^{sX}(e^{tU}(Hu))$,
\[
\begin{split}
\lim_{\Delta s \to 0} & \frac{\tilde{H}(s+ \Delta s,t)f  -\tilde{H}(s,t)f}{\Delta s}\\
=&\lim_{\Delta s \to 0} \left(  \frac{e^{(s+\Delta s)X}e^{tU}\left (\frac{\tilde{H(t)}H^{-1}}{t^{\beta}}\right)e^{-tU}e^{-(s+ \Delta s)X}f}{\Delta s} \right.\\
-& \left. \frac{e^{sX}e^{tU}\left (\frac{\tilde{H(t)}H^{-1}}{t^{\beta}}\right )e^{-tU}e^{-sX}f}{\Delta s}  \right)\\
=&\lim_{\Delta s \to 0} \left ( \frac{e^{(s+\Delta s)X}e^{tU}\left (\frac{\tilde{H(t)}H^{-1}}{t^{\beta}}\right )e^{-tU}e^{-(s+ \Delta s)X}f}{\Delta s} \right.\\
\pm&  \frac{e^{(s+\Delta s)X}e^{tU}\left (\frac{\tilde{H(t)}H^{-1}}{t^{\beta}}\right )e^{-tU}e^{-sX}f}{\Delta s}\\
-& \left. \frac{e^{sX}e^{tU}\left (\frac{\tilde{H(t)}H^{-1}}{t^{\beta}}\right )e^{-tU}e^{-sX}f}{\Delta s} \right)\\
=& \lim_{\Delta s \to 0} \left( \frac{(e^{(s+\Delta s)X})e^{tU}\left (\frac{\tilde{H(t)}H^{-1}}{t^{\beta}}\right )e^{-tU}(e^{-sX}-e^{-(s + \Delta)X})f}{\Delta s} \right. 
\end{split}
\]\[
\begin{split}+& \left. \frac{(e^{(s+\Delta s)X}-e^{sX})e^{tU}\left (\frac{\tilde{H(t)}H^{-1}}{t^{\beta}}\right )e^{-tU}e^{-sX}f}{\Delta s} \right )\\
=&\lim_{\Delta s \to 0} \left( \frac{e^{sX}e^{\Delta s X}e^{tU}\left (\frac{\tilde{H(t)}H^{-1}}{t^{\beta}}\right )e^{-tU}e^{-sX}(I-e^{-\Delta sX})f}{\Delta s}\right. \\
+& \left. \frac{e^{sX}(e^{\Delta s X}-I)e^{tU}\left (\frac{\tilde{H(t)}H^{-1}}{t^{\beta}}\right )e^{-tU}e^{-sX}f}{\Delta s} \right)\\
 \pm &\lim_{\Delta s \to 0} e^{sX}e^{\Delta s X} e^{tU} \left ( \frac{\tilde{H(t)}H^{-1}}{t^{\beta}}  \right ) e^{-tU}e^{-sX}Xf \\
=&\lim_{\Delta s \to 0} e^{sX}e^{\Delta s X}e^{tU}\left(\frac{\tilde{H(t)}H^{-1}}{t^{\beta}}\right )e^{-tU}e^{-sX}\left(\frac{I-e^{-\Delta sX}}{\Delta s}-X\right)f\\
+  & \lim_{\Delta s \to 0} e^{sX}e^{\Delta s X}e^{tU}\left (\frac{\tilde{H(t)}H^{-1}}{t^{\beta}}\right )e^{-tU}e^{-sX}Xf\\
+ &\lim_{\Delta s \to 0} e^{sX}   \left ( \frac{e^{\Delta s X}-I}{\Delta s} \right)  e^{tU} \left (\frac{\tilde{H(t)}H^{-1}}{t^{\beta}}\right )e^{-tU}e^{-sX}f.
\end{split}
\]
\noindent  The first limit,
\[
\lim_{\Delta s \to 0} e^{sX}e^{\Delta s X}e^{tU}\left(\frac{\tilde{H(t)}H^{-1}}{t^{\beta}}\right )e^{-tU}e^{-sX}\left(\frac{I-e^{-\Delta sX}}{\Delta s}-X\right )f = 0,
\]
\noindent follows from the preliminary assumption,
\[
\lim_{\Delta s \to 0} \left ( \frac{e^{\Delta s}-I}{\Delta s}\right ) = X.
\]
\noindent and
\[
\lim_{\Delta s \to 0} e^{\Delta s X} = I.
\]
\noindent The second and third limits follow from above,
\[
\begin{split}
\lim_{\Delta s \to 0} e^{sX}e^{\Delta s X}e^{tU}\left (\frac{\tilde{H(t)}H^{-1}}{t^{\beta}}\right )e^{-tU}e^{-sX}Xf\\
=& \> e^{sX}e^{tU}\left (\frac{\tilde{H(t)}H^{-1}}{t^{\beta}}\right )e^{-tU}e^{-sX}Xf
\end{split}
\]
\noindent and
\[
\begin{split}
 \lim_{\Delta s \to 0} e^{sX}   \left ( \frac{e^{\Delta s X}-I}{\Delta s} \right )  e^{tU} \left (\frac{\tilde{H(t)}H^{-1}}{t^{\beta}}\right )e^{-tU}e^{-sX}f\\
 =    \> e^{sX} X & e^{tU} \left (\frac{\tilde{H(t)}H^{-1}}{t^{\beta}}\right )e^{-tU}e^{-sX}f.
\end{split}
\]
\noindent Since $f=e^{sX}(e^{tU}(Hu))$,
\[
\begin{split}
e^{sX}Xe^{tU}\tilde{H(t)}H^{-1}e^{-tU}e^{-sX}f\\
=& \> e^{sX}Xe^{tU}\tilde{H(t)}H^{-1}e^{-tU}e^{-sX}e^{sX}(e^{tU}(Hu))\\
=& \> e^{sX}Xe^{tU}H(t)u \in \mathcal{H}
\end{split}
\]
\noindent and
\[
\begin{split}
e^{sX}e^{tU}\tilde{H(t)}H^{-1}e^{-tU}Xe^{-sX}f\\
=& \> e^{sX}e^{tU}\tilde{H(t)}H^{-1}e^{-tU}Xe^{-sX}e^{sX}e^{tU}(Hu)\\
=& \> e^{sX}e^{tU}\tilde{H(t)}H^{-1}e^{-tU}Xe^{tU}(Hu) \in \mathcal{H}
\end{split}
\]
\noindent since $e^{tU}(Hu) \subset D$ and $\tilde{H(t)}H^{-1}$ is a bounded operator on $\mathcal{H}$.


Since $Ran(U)$ is dense in $\mathcal{H}$ and $e^{tU}$ and $e^{sX}$ are bounded, invertible operators, $e^{sX}\left (e^{tU}(Ran(H))\right )$ is dense in $\mathcal{H}$. Thus,
\[
\frac{\partial \tilde{H(s,t)}}{\partial s} =\frac{1}{t^{\beta}}e^{sX}\left [X,e^{tU}\tilde{H(t)}H^{-1}e^{-tU} \right ]e^{-sX}
\]
\noindent is defined on $\mathcal{H}$.
%

Now we can rewrite
\[
\int_{0}^{\sigma}e^{sX}e^{tU}\left (H-\frac{H(t)}{t^{\beta}} \right )gds = \int_{0}^{\sigma}e^{sX}e^{tU}\left (I-\frac{H(t)H^{-1}}{t^{\beta}} \right )Hg \> ds
\]
\noindent and again consider the extension
\[
\int_{0}^{\sigma}e^{sX}e^{tU}\left (I-\frac{\tilde{H(t)}H^{-1}}{t^{\beta}} \right )Hgds = \int_{0}^{\sigma}\tilde{H}(s,t)e^{sX}e^{tU}f \> ds.
\]


Integration by parts gives
\[
\begin{split}
\int_{0}^{\sigma}\tilde{H}(s,t)e^{sX}e^{tU}fds\\
=&\tilde{H}(\sigma,t)\int_{0}^{\sigma}e^{sX}e^{tU}fds
-\int_{0}^{\sigma}\frac{\partial  \tilde{H}(S,t)}{\partial S}\left [\int_{0}^{S}e^{sX}e^{tU}fds \right ]dS.
\end{split}
\]


So now we must ensure that
\[
\frac{\partial \tilde{H}(s,t)}{\partial s} = \frac{1}{t^{\beta}}e^{sX}\left [X,e^{tU}\tilde{H(t)}H^{-1}e^{-tU} \right ]e^{-sX}
\]
\noindent is uniformly bounded in $\Vert \cdot \Vert_{op}$. Let $h \in \mathcal{H}$.
\[
\begin{split}
\frac{1}{t^{\beta}}e^{sX}\left (\left [X,e^{tU}\tilde{H(t)}H^{-1}e^{-tU} \right ] \right )e^{-sX}h\\
=& \> \frac{1}{t^{\beta}}e^{sX}\left (\left [X,e^{tU}\right ]\tilde{H(t)}H^{-1}e^{-tU}\right.\\
+& \> e^{tU}\left [X,\tilde{H(t)}H^{-1} \right ]e^{-tU}\\
+& \left. e^{tU}\tilde{H(t)}H^{-1}\left [X,e^{-tU}\right ]\right ) e^{-sX}h.
\end{split}
\]

Using the identity $e^{-tU}\left [X,e^{tU}\right ]=-\left [X,e^{-tU}\right ]e^{tU}$ we can simplify and combine terms:
\[
\frac{1}{t^{\beta}}e^{sX}e^{tU}\left (H(t)\tilde{H(t)}H^{-1}+\left [X,\tilde{H(t)}H^{-1}\right ]-\tilde{H(t)}H^{-1}H(t) \right )e^{-tU}e^{-sX}h
\]
\[
= e^{sX}e^{tU}\left (\left [X,\frac{\tilde{H(t)}H^{-1}}{t^{\beta}}\right ]-\frac{\tilde{H(t)}H^{-1}}{t^{\beta}}\left [\tilde{H(t)},H \right ]H^{-1}  \right )e^{-tU}e^{-sX}h
\]
\noindent where $\left [X,\frac{\tilde{H(t)}H^{-1}}{t^{\beta}} \right ]$ and $\left [\tilde{H(t)},H \right ]H^{-1}$ are the bounded extensions of $\left [X,\frac{H(t)H^{-1}}{t^{\beta}} \right]$ and $\left [H(t),H \right ]H^{-1}$.


Conditions $(i)$, $(ii)$, and $(iii)$ imply that
\[
\begin{split}
\left \Vert \frac{\partial \tilde{H}(s,t)}{\partial s} h \right \Vert_{\mathcal{H}}\\
\leq & C\left (\left \Vert   \frac{\tilde{H(t)}}{t^{\beta}} H^{-1} \right \Vert_{op} \cdot \> \left \Vert [\tilde{H(t)}, H]H^{-1} \right \Vert_{op}
+ \left \Vert [X,\frac{\tilde{H(t)}}{t^{\beta}}H^{-1}] \right \Vert_{op}\right ) \cdot \left \Vert h \right \Vert_{\mathcal{H}}\\
\leq & C_{2}\cdot \Vert h \Vert_{\mathcal{H}}
\end{split}
\]
\noindent for some constants $C$ and $C_{2}$, and thus,
\[
\left \Vert \frac{\partial \tilde{H}(s,t)}{\partial s}  \right \Vert_{op} \leq C_{2}.
\]

\noindent II. \hspace{0.5cm} For $g \in Dom(H)$,
\[
\frac{1}{t^{\beta}}\int_{0}^{\sigma}e^{sX}e^{tU}H(t)gds= \frac{1}{t^{\beta}}\int_{0}^{\sigma}e^{sX}e^{tU}Xgds - \frac{1}{t^{\beta}}\int_{0}^{\sigma}\frac{d}{ds}e^{sX}e^{tU}gds
\]
\noindent which implies that
\[
\left \Vert \frac{1}{t^{\beta}}\int_{0}^{\sigma}e^{sX}e^{tU}H(t)gds  \right \Vert_{\mathcal{H}} \leq \frac{C_{3}}{t^{\beta}} \left \Vert Xg \right \Vert_{\mathcal{H}} + \frac{C_{4}}{t^{\beta}}\left \Vert g \right \Vert_{\mathcal{H}}.
\]
\noindent Finally, from I. and II.,
\[
\begin{split}
\sup_{s \in [0,\sigma]} \Vert \int_{0}^{s}e^{sX}e^{tU}fds \Vert_{\mathcal{H}}
\leq & \sup_{s\in [0,\sigma]} \left (\left \Vert \tilde{H}(s,t) \right \Vert_{op} \cdot \left \Vert \int_{0}^{s} e^{sX}e^{tU}fds \right \Vert_{\mathcal{H}} \right )\\
 + & \> \sigma \cdot \sup_{s\in [0,\sigma]} \left (\left \Vert \frac{\partial \tilde{H}(s,t)}{\partial s} \right \Vert_{op} \cdot \left \Vert \int_{0}^{s} e^{sX}e^{tU}fds \right \Vert_{\mathcal{H}} \right )\\
 + & \frac{C_{3} \left \Vert Xg \right \Vert_{\mathcal{H}} +  C_{4}\left \Vert g \right \Vert_{\mathcal{H}}}{t^{\beta}}.
\end{split}
\]
\noindent So for $\sigma>0$, chosen such that $0 < C_{1}+ \sigma C_{2} < 1$, for all $t$ sufficiently large,
\[
\sup_{s\in[0,\sigma]} \left \Vert \int_{0}^{s}e^{sX}e^{tU}fds \right \Vert_{\mathcal{H}} \leq \frac{C_{3}\left \Vert Xg \right \Vert_{\mathcal{H}} + C_{4} \left \Vert g \right \Vert_{\mathcal{H}}}{t^{\beta}}\frac{1}{(1-C_{1} - \sigma  C_{2})} = O(\frac{1}{t^{\beta}}).
\]


\noindent
Thus, since $\hat{\mu}_{f}(t) \in L^{2}(\mathbb{R})$, $\mu_{f}$ is absolutely continuous. As this holds for $f \in Ran(H)$, $\mu_{f}$ is absolutely continuous for any $f \in \mathcal{H}$.


\begin{remark}
The proof for the discrete case is the same aside from the replacement of the continuous parameter $t$ and norm $\Vert \cdot \Vert_{L^{2}(\mathbb{R})}$ by the discrete parameter $n$ and norm $\Vert \cdot \Vert_{\ell^{2}(\mathbb{Z})}$. The conclusion becomes
\[
\left \Vert \; \langle e^{nU}f,f \rangle_{\mathcal{H}}\; \right \Vert_{\ell^{2}(\mathbb{Z})}=O(\frac{1}{n^{\beta}})
\]
\noindent \textit{for $\beta > \frac{1}{2}$, and thus, $\mu_{f}(n) \in \ell^{2}(\mathbb{Z})$.}

\end{remark}
\end{proof}

\section{Applications to flows}

\subsection{Time-changes of unipotent flows on
 homogeneous spaces of semisimple groups}
\noindent As a direct consequence of Theorem \ref{main}, we derive a result for a specific category of generating operators.


Let $G$ be a semisimple Lie group and let the manifold $M = \Gamma \setminus G$ for some lattice $\Gamma$ in $G$ such that $M$ has finite area.


By the Jacobson-Morozov Theorem, any nilpotent element $U$ of the semisimple Lie algebra of $G$ is contained in a subalgebra isomorphic
to $\mathfrak{sl}_{2}$. This means that this subalgebra contains an element $X$, such that $[U,X]=U$. Let $e^{tU}$ be a unitary operator of the Hilbert space $L^{2}(M,vol)$.  Thus, if the unipotent flow generated by $U$, $f \circ \phi_{t}^{U} = e^{tU}f$, $f \in L^{2}(M,vol)$, is ergodic, then from Lemma 5.1 in  \cite{moore}, it has purely absolutely continuous spectrum on
\[
L^{2}_{0}(M,vol)=\left \{ f \in L^{2}(M) \> | \int_{M} f \> vol=0 \right \}.
\]
\noindent Let $\tau: M \times \mathbb{R} \to \mathbb{R}$ such that  $\tau \in C^{\infty}(M, \mathbb{R})$ and
\[
\tau(x,t+t')=\tau(x,t)+\tau(\phi_{t}^{U}(x),t').
\]
\noindent Let $\alpha: M \to \mathbb{R}^{+}$, be the infinitesimal generator of $\tau$, such that $\alpha \in C^{\infty}(M)$ and
\[
\int_{M} \alpha vol = \int_{M} vol_{\alpha}=1
\]
where $vol$ is the $\phi_{t}^{U}$-invariant volume form and $vol_{\alpha}$ is the $\phi_{t}^{U_{\alpha}}$-invariant volume form.
\vspace{0.2cm}
\noindent Now we consider a time-changed flow, $\{\phi^{U_{\alpha}}_{t}\}$, generated by
\[
U_{\alpha}=:U / \alpha.
\]
\noindent Let $e^{tU_{\alpha}}$ be a unitary operator on the Hilbert space $L^{2}(M, vol_{\alpha})$. The following formulas hold on $D = C^{\infty}(M)$.
\[
[X,U_{\alpha}]=G(\alpha)U_{\alpha}=(\frac{X\alpha}{\alpha}-1)U_{\alpha}=H
\]
\[
e^{-tU_{\alpha}}\left [X,e^{tU_{\alpha}} \right ]=G(\alpha, t)U_{\alpha}=\boxed{\left (\int_{0}^{t} (\frac{X\alpha}{\alpha}-1) \circ \phi_{\tau}^{U_{\alpha}}(x) d\tau \right )U_{\alpha}=H(t).}
\]

%
The ergodicity of $\phi_{t}^{U_{\alpha}}$ gives us the following limit a.e.,
\[
\begin{split}
\lim_{t \to \infty} \frac{G(\alpha,t)}{t}
=& \lim_{t \to \infty} \frac{1}{t}\int_{0}^{t} \left (\frac{X\alpha}{\alpha}-1 \right ) \circ \phi_{\tau}^{U_{\alpha}}(x) d\tau\\
= & \lim_{t \to \infty}\left ( \frac{1}{t}\int_{0}^{t} \frac{X\alpha}{\alpha} \circ \phi_{\tau}^{U_{\alpha}}(x) d\tau +\frac{1}{t}\int_{0}^{t} -1 \circ \phi_{\tau}^{U_{\alpha}}(x) d\tau \right )\\
=& \lim_{t \to \infty}\left ( \frac{1}{t}\int_{0}^{t} \frac{X\alpha}{\alpha} \circ \phi_{\tau}^{U_{\alpha}}(x) d\tau -1 \right )\\
=&\int_{M} \frac{X\alpha}{\alpha}\> dvol_{\alpha} -1=-1.
\end{split}
\]
\noindent From the Dominated Convergence Theorem, with dominating function $2 \Vert G(\alpha) \Vert_{\infty}$, we have convergence in $L^{2}(M)$. Thus, for $u \in C^{\infty}(M)$,
\[
\lim_{t \to \infty} \frac{G(\alpha,t)}{t}U_{\alpha} u   = \boxed{ -U_{\alpha}u=Hu.}
\]
\noindent Lastly,
\[
\begin{split}
\int_{M}e^{tU_{\alpha}}f \> \cdot \overline{f} \> dvol_{\alpha}
=  \int_{M}e^{tU_{\alpha}}f \> \cdot \overline{\alpha f} \> dvol\\
 = &\int_{M} e^{sX}e^{tU_{\alpha}}f \> \cdot \overline{e^{sX}\alpha f} \> dvol\\
 = & \int_{M} e^{sX}e^{tU_{\alpha}}f \> \cdot \overline{\frac{1}{\alpha}e^{sX}\alpha f} \> dvol_{\alpha}.
 \end{split}
\]
\noindent So if we integrate both sides  of
\[
\langle e^{tU_{\alpha}}f,f \rangle _{L^{2}(M, vol_{\alpha})} = \langle e^{sX}e^{tU_{\alpha}}f, \frac{1}{\alpha}e^{sX}\alpha f \rangle _{L^{2}(M, vol_{\alpha})}
\]
\noindent with respect to $s$, we obtain the following equality
\[
\begin{split}
\Vert \; \langle e^{tU_{\alpha}}f,f \rangle _{L^{2}(M, vol_{\alpha})}\; \Vert_{L^{2}(\mathbb{R}, dt)}\; \\
= & \; \Vert \frac{1}{\sigma}\int_{0}^{\sigma} \langle e^{sX}e^{tU_{\alpha}}f, \frac{1}{\alpha}e^{sX}\alpha f \rangle _{L^{2}(M, vol_{\alpha})}\> \> ds \> \>  \Vert_{L^{2}(\mathbb{R}, dt)}.
\end{split}
\]
\noindent Thus, the preliminary assumptions for Theorem \ref{main}  are satisfied with $B_{1}=\frac{1}{\alpha}I$ and $B_{2}=\alpha I$.


\begin{theorem} \label{unipotent}

\noindent \textbf{a.} Any smooth time-change of an ergodic flow on $M$ generated by a nilpotent element of a semisimple Lie algebra  has absolutely continuous spectrum on $L^{2}_{0}(M,vol_{\alpha})$ if $\Vert  \frac{X\alpha}{\alpha} \Vert_{\infty} < 1 $.

\vspace{0.2cm}

\noindent \textbf{b.} Any smooth time-change of a uniquely ergodic flow on $M$ generated by a nilpotent element of a semisimple Lie algebra  has absolutely continuous spectrum on $L^{2}_{0}(M,vol_{\alpha})$.

\end{theorem}

\begin{remark} The condition in part \textbf{a.} is equivalent to the condition employed by Kushnirenko \cite{kush} (Theorem 2) to prove mixing for the time-changes of the horocycle flow. As shown by Marcus, this condition is unnecessarily restrictive. In the compact setting, the authors in \cite{forniulci} prove spectral results using the implicit unique ergodicity instead of requiring such a condition. The author in \cite{tiedrahor} proves similar spectral results by imposing this Kushnirenko condition; the author later substitutes this condition by a utilization of unique ergodicity \cite{tiedra} in the compact case. In the noncompact setting it remains open as to whether or not spectral results can be derived without imposing a Kushnirenko-type condition.
\end{remark}

\begin{proof}


\noindent We show that the conditions of Theorem \ref{main} hold.

\vspace{0.2cm}

\noindent \textbf{a.}
\noindent $(i)$
\noindent Let $f=U_{\alpha}g$ for $g \in C^{\infty}(M)$.
\[
\begin{split}
\left \Vert \frac{H(t)}{t}H^{-1}f \right \Vert_{L^{2}(M, vol_{\alpha})}
=&\left \Vert \frac{G(\alpha,t)}{t}U_{\alpha}(-U_{\alpha}^{-1}f) \right\Vert_{L^{2}(M, vol_{\alpha})}\\
=&\left \Vert \frac{G(\alpha,t)}{t}f \right \Vert_{L^{2}(M, vol_{\alpha})}\\
\leq & \> 2 \left \Vert G(\alpha) \right \Vert_{\infty} \cdot \Vert f \Vert_{L^{2}(M, vol_{\alpha})}\\
\leq & \> 2 \Vert f \Vert_{L^{2}(M, vol_{\alpha})}
\end{split}
\]
\noindent Since the above holds for $f \in Ran(U_{\alpha})$, $\frac{H(t)}{t}H^{-1}$ extends to  a bounded operator on $\overline{Ran(U_{\alpha})}=L^{2}_{0}(M,vol_{\alpha})$ with uniformly bounded norm in $t$,

\ \vspace*{-10pt}
\[
\left \Vert \frac{H(t)}{t}H^{-1} \right \Vert_{op}\leq2.
\]


Also,
\[
\begin{split}
\left \Vert (I - \frac{H(t)}{t}H^{-1})f \right  \Vert_{L^{2}(M,vol_{\alpha})}
=&\left \Vert \left (1+\frac{G(\alpha,t)}{t} \right )f  \right \Vert_{L^{2}(M,vol_{\alpha})}\\
=&\left \Vert \left (1+ \left (\frac{1}{t}\int_{0}^{t} (\frac{X\alpha}{\alpha}-1) \circ \phi_{\tau}^{U_{\alpha}}(x) d\tau \right ) \right )f \right \Vert_{L^{2}(M,vol_{\alpha})}\\
=&\left \Vert \left (\frac{1}{t}\int_{0}^{t} \frac{X\alpha}{\alpha} \circ \phi_{\tau}^{U_{\alpha}}(x) d\tau \right )f \right  \Vert_{L^{2}(M,vol_{\alpha})}\\
\leq & \left \Vert \left (\frac{1}{t}\int_{0}^{t} \frac{X\alpha}{\alpha} \circ \phi_{\tau}^{U_{\alpha}}(x) d\tau \right ) \right \Vert_{\infty} \cdot \Vert f \Vert_{L^{2}(M,vol_{\alpha})}\\
\leq & \left \Vert \frac{X\alpha}{\alpha} \right \Vert_{\infty}\cdot \left \Vert f \right \Vert_{L^{2}(M,vol_{\alpha})}\\
\leq & \> \Vert f \Vert_{L^{2}(M,vol_{\alpha})}.
\end{split}
\]
 \noindent Since the above holds on $Ran(U_{\alpha})$ the following is true on $\overline{Ran(U_{\alpha})}=L^{2}_{0}(M,vol_{\alpha}),$
\[
\limsup_{t \to \infty} \left \Vert I - \frac{H(t)}{t}H^{-1} I \right \Vert_{op}=\limsup_{t \to \infty} \left \Vert I + \frac{G(\alpha,t)}{t} I \right \Vert_{op} < 1.
\]


\noindent $\textit{(ii)}$ In the following calculation we use that
\[
D\phi_{t}^{U_{\alpha}}(X)=G(\alpha,t)U_{\alpha}\circ \phi_{t}^{U_{\alpha}}+ X \circ \phi_{t}^{U_{\alpha}}
\]
\noindent where $D\phi_{t}^{U_{\alpha}}$ denotes the differential of the diffeomorphism $\phi_{t}^{U_{\alpha}}$.
\[
\begin{split}
\left [X, \frac{H(t)}{t}H^{-1} \right ]
=& X\left (\frac{1}{t}\int_{0}^{t}(\frac{X\alpha}{\alpha}-1)\circ \phi_{\tau}^{U_\alpha} d\tau \right )-\left (\frac{1}{t}\int_{0}^{t}(\frac{X\alpha}{\alpha}-1)\circ \phi_{\tau}^{U_\alpha} d\tau \right )X\\
=&\frac{1}{t}\int_{0}^{t}\left (D\phi_{\tau}^{U_{\alpha}}(X) \circ \phi_{-\tau}^{U_\alpha} \right )\left (\frac{X\alpha}{\alpha} \right ) \circ \phi_{\tau}^{U_\alpha} d\tau\\
=&\frac{1}{t}\int_{0}^{t}\left (G(\alpha,\tau)U_{\alpha}+X \right )\left (\frac{X\alpha}{\alpha} \right ) \circ \phi_{\tau}^{U_\alpha} d\tau\\
=&\frac{1}{t}\int_{0}^{t}G(\alpha,\tau)U_{\alpha}\left (\frac{X\alpha}{\alpha} \right )\circ \phi_{\tau}^{U_\alpha}(x)d\tau  + \frac{1}{t}\int_{0}^{t}X \left (\frac{X\alpha}{\alpha} \right )\circ \phi_{\tau}^{U_\alpha}(x)d\tau\\
=&\frac{1}{t}\int_{0}^{t}G(\alpha,\tau)\frac{d}{d\tau}\left (\frac{X\alpha}{\alpha} \right )\circ \phi_{\tau}^{U_\alpha}(x)d\tau  + \frac{1}{t}\int_{0}^{t}X \left (\frac{X\alpha}{\alpha} \right )\circ \phi_{\tau}^{U_\alpha}(x)d\tau
\end{split}
\]
\noindent We integrate
\[
\frac{1}{t}\int_{0}^{t}G(\alpha,\tau)\frac{d}{d\tau}\left (\frac{X\alpha}{\alpha} \right )\circ \phi_{\tau}^{U_{\alpha}}(x)d\tau
\]
\noindent by parts,
\[
\begin{split}
\frac{1}{t}\int_{0}^{t}G(\alpha,\tau)\frac{d}{d\tau}\left (\frac{X\alpha}{\alpha} \right )\circ \phi_{\tau}^{U_{\alpha}}(x)d\tau\\
= \> \frac{G(\alpha,t)}{t} \left (\frac{X\alpha}{\alpha} \right )&\circ \phi_{t}^{U_\alpha} - \frac{1}{t}\int_{0}^{t} (\frac{X\alpha}{\alpha}-1)(\frac{X\alpha}{\alpha}) \circ \phi_{\tau}^{U_{\alpha}}(x)d\tau,
\end{split}
\]
\noindent and obtain the bound,
\[
\begin{split}
&\left \Vert \frac{1}{t}\int_{0}^{t}(G(\alpha,\tau)U_{\alpha}+X)\left (\frac{X\alpha}{\alpha} \right ) \circ \phi_{\tau}^{U_{\alpha}} d\tau \right \Vert_{\infty}
\\\leq & \left \Vert \frac{G(\alpha,t)}{t}\left (\frac{X\alpha}{\alpha}\right )\right \Vert_{\infty}\\
+& \left \Vert \frac{1}{t}\int_{0}^{t} (\frac{X\alpha}{\alpha}-1)\left (\frac{X\alpha}{\alpha} \right ) \circ \phi_{\tau}^{U_{\alpha}}(x)d\tau \right \Vert_{\infty}\\
+& \left \Vert \frac{1}{t}\int_{0}^{t}X\left (\frac{X\alpha}{\alpha}\right )\circ \phi_{\tau}^{U_{\alpha}}(x)d\tau \right \Vert_{\infty}\\
\leq & \> 2 \> \cdot \left \Vert \frac{X\alpha}{\alpha}-1 \right \Vert_{\infty} \cdot \left \Vert \left (\frac{X\alpha}{\alpha} \right ) \right \Vert_{\infty}\\
+ & \left \Vert X\left (\frac{X\alpha}{\alpha} \right) \right \Vert_{\infty}\\
\leq& \> 2(2) + C_{\alpha}''
\end{split}
\]
\noindent where $C_{\alpha}''$ depends on the second derivative of $\alpha$.


 Since
\[
\left [X, \frac{H(t)}{t}H^{-1} \right ]
\]
\noindent is the multiplication operator given by
\[
\left (\frac{1}{t}\int_{0}^{t}(G(\alpha,\tau)U_{\alpha}+X)\left (\frac{X\alpha}{\alpha}\right ) \circ \phi_{\tau}^{U_{\alpha}} d\tau \right )\cdot I,
\]
\noindent we obtain the following bound,
\[
\begin{split}
&\left \Vert \left [X, \frac{H(t)}{t}H^{-1} \right ] f \right \Vert_{L^{2}(M,vol_{\alpha})}\\
\leq & \left \Vert \frac{1}{t}\int_{0}^{t}(G(\alpha,\tau)U_{\alpha}+X)\left (\frac{X\alpha}{\alpha} \right ) \circ \phi_{\tau}^{U_{\alpha}} d\tau \right \Vert_{\infty} \cdot \Vert f
\Vert_{L^{2}(M,vol_{\alpha})}\\
 \leq & (4 + C_{\alpha}'') \cdot \Vert f
\Vert_{L^{2}(M,vol_{\alpha})}.
\end{split}
\]
\noindent Thus, $\left [X, \frac{H(t)}{t}H^{-1} \right ]$ extends to a bounded operator on $\overline{Ran(U_{\alpha})}$ with operator norm uniformly bounded in $t$:
\[
\left \Vert \left [X, \frac{H(t)}{t}H^{-1} \right ] \right \Vert_{op} \leq 4 + C_{\alpha}''.
\]

\noindent $\textit{(iii)}$
\[
\begin{split}
\left \Vert \left [H(t),H \right ]H^{-1}f \right \Vert_{L^{2}(M, vol_{\alpha})}
=&\left \Vert \left [G(\alpha,t)U_{\alpha},-U_{\alpha} \right ](-U_{\alpha}^{-1}f) \right \Vert_{L^{2}(M, vol_{\alpha})}\\
=&\left \Vert \left [U_{\alpha}(\int_{0}^{t} (\frac{X\alpha}{\alpha}-1) \circ \phi_{\tau}^{U_{\alpha}} d\tau) \right ] \cdot f \right \Vert_{L^{2}(M, vol_{\alpha})}\\
=& \left \Vert \left [\int_{0}^{t} \frac{d}{d\tau}(\frac{X\alpha}{\alpha}-1) \circ \phi_{\tau}^{U_{\alpha}} d\tau) \right ] \cdot f \right \Vert_{L^{2}(M, vol_{\alpha})}\\
\leq & \> \Vert G(\alpha)\circ \phi_{t}^{U_{\alpha}} -G(\alpha) \Vert_{\infty} \cdot \Vert f \Vert_{L^{2}(M, vol_{\alpha})}\\
\leq &  \> 2 \Vert G(\alpha) \Vert _{\infty} \cdot \Vert f \Vert_{L^{2}(M, vol_{\alpha})}\\ \leq & \> 2(2) \cdot \Vert  f \Vert_{L^{2}(M, vol_{\alpha})}
\end{split}
\]


The above holds on coboundaries of the form $f=U_{\alpha}g$,  so on $\overline{Ran(U_\alpha)}=L^{2}_{0}(M, vol_{\alpha})$,
\[
\left \Vert [H(t),H]H^{-1} \right \Vert_{op} \leq 4.
\]

Since conditions $(i)$, $(ii)$, and $(iii)$ of Theorem \ref{main}. are satisfied on $Ran(U_{\alpha})$, the time-changed flow, $\{\phi^{U_{\alpha}}_{t}\}$, has purely absolutely continuous spectrum on $\overline{Ran(U_{\alpha})}$. Since $\{\phi^{U_{\alpha}}_{t}\}$ is ergodic,  $\overline{Ran(U_{\alpha})}=L^{2}_{0}(M,vol_{\alpha})$.

\vspace{0.2cm}
%
This concludes the proof of part \textbf{a}.

\vspace{0.2cm}
\noindent \textbf{b}. Now we assume that the flow $\{ \phi_{t}^{U} \}$, and hence $\{ \phi_{t}^{U_{\alpha}} \}$, are uniquely ergodic.
\vspace{0.2cm}\[
\begin{split}
&\left \Vert \left (I - \frac{H(t)}{t}H^{-1} \right )f \>  \right \Vert_{L^{2}(M,vol_{\alpha})}\\[2mm]
=&\left \Vert \left(1+\frac{G(\alpha,t)}{t} \right )f \> \right \Vert_{L^{2}(M,vol_{\alpha})}\\[2mm]
=&\left \Vert \left (1+ \left (\frac{1}{t}\int_{0}^{t} (\frac{X\alpha}{\alpha}-1) \circ \phi_{\tau}^{U_{\alpha}}(x) d\tau \right ) \right )f \> \right \Vert_{L^{2}(M,vol_{\alpha})}\\[2mm]
=&\left \Vert \left (\frac{1}{t}\int_{0}^{t} \frac{X\alpha}{\alpha} \circ \phi_{\tau}^{U_{\alpha}}(x) d\tau \right ) f \> \right \Vert_{L^{2}(M,vol_{\alpha})}\\[2mm]
\leq & \left \Vert \left (\frac{1}{t}\int_{0}^{t} \frac{X\alpha}{\alpha} \circ \phi_{\tau}^{U_{\alpha}}(x) d\tau \right ) \right \Vert_{\infty} \cdot \Vert f \Vert_{L^{2}(M,vol_{\alpha})}.
\end{split}
\]
\noindent If $\{ \phi_{t}^{U_{\alpha}} \}$ is uniquely ergodic, then the following converges uniformly,
\vspace{0.2cm}\[
\lim_{t \to \infty} \frac{1}{t}\int_{0}^{t} \frac{X\alpha}{\alpha}\circ \phi_{\tau}^{U_{\alpha}}d\tau = \int_{M}\frac{X\alpha}{\alpha} \> dvol_{\alpha} = 0,
\]
\noindent and thus,
\vspace{0.2cm}\[
\begin{split}
\limsup_{t \to \infty} \left \Vert I + \frac{H(t)}{t}H^{-1}  \right \Vert_{op}
\leq & \limsup_{t \to \infty} \left \Vert \left (\frac{1}{t}\int_{0}^{t} \frac{X\alpha}{\alpha} \circ \phi_{\tau}^{U_{\alpha}}(x) d\tau \right ) \right \Vert_{\infty}\\[2mm]
=&\left \Vert \int_{M}\frac{X\alpha}{\alpha} \> dvol_{\alpha} \right \Vert_{\infty}=0.
\end{split}
\]
\noindent Hence,
\vspace{0.2cm}\[
\limsup_{t \to \infty} \left \Vert I + \frac{H(t)}{t}H^{-1}  \right \Vert_{op}<1
\]
\noindent is satisfied on $\overline{Ran(U_{\alpha})}=L^{2}_{0}(M,vol_{\alpha})$ without imposing any further conditions on $\frac{X\alpha}{\alpha}$. The remainder of the proof is the same as in \textbf{a} except that
\[
\left \Vert \frac{X\alpha}{\alpha} \right \Vert_{\infty} \leq  C_{\alpha}' < \infty
\]
\noindent where $C_{\alpha}'$ is finite but not necessarily equal to $1$.
\end{proof}

\begin{theorem}[Maximal Spectral Type]
The maximal spectral type of the uniqely ergodic flow $\{\phi^{U_{\alpha}}_{t}\}$ is Lebesgue on the subspace $\overline{Ran(U_{\alpha})}$.
\end{theorem}

\begin{proof}

\noindent We follow the method in \cite{forniulci}.

\begin{lemma}\label{lem} \cite{forniulci} Suppose that the maximal spectral type of $\{ \phi^{U_{\alpha}}_{t} \}$ is not Lebesgue. Then there exists a smooth non-zero function $\omega \in L^{2}(\mathbb{R},dt)$ such that for all functions $g \in C^{\infty}(M)$ the following holds:
$$
\int_{\mathbb{R}} \omega(t) \int^{\sigma}_{0} e^{sX}e^{tU_{\alpha}}U_{\alpha}g\> ds\> dt=0
$$

\end{lemma}

\begin{proof}
Since the maximal spectral type is not Lebesgue, then there exists a compact set $A \subset \mathbb{R}$ such that $A$ has positive Lebesgue measure but measure $0$ with respect to the maximal spectral type. So we let $\omega \in L^{2}(\mathbb{R})$ be the complex conjugate of the Fourier transform of the characteristic function $\chi_{A}$ of the set $A \subset \mathbb{R}$. For $f,h \in Ran(U_{\alpha})$, let $\mu_{f,h}$ denote the joint spectral measure (which we know is absolutely continuous with respect to Lebesgue since $f, h \in Ran(U_{\alpha})$. Thus,
$$
\int_{\mathbb{R}} \omega(t)\langle e^{tU_{\alpha}}f, h \rangle _{L^{2}(M,vol)}dt=\int_{\mathbb{R}} \chi_{A}(\xi)d\mu_{f,h}(\xi)=0.
$$


\noindent In particular, when $f=U_{\alpha}g$ we have
\[
\begin{split}
\int_{0}^{\sigma} \int_{\mathbb{R}} \omega(t) \langle e^{sX}e^{tU_{\alpha}}U_{\alpha}g,h\rangle _{L^{2}(M,vol)}\> dt\> ds\\
= \langle \int_{\mathbb{R}} \omega(t) & \int_{0}^{\sigma} e^{sX}e^{tU_{\alpha}}U_{\alpha}g\>ds\>dt,h\rangle_{L^{2}(M,vol)}=0.
\end{split}
\]


\noindent Recall that satisfying conditions $(i)$ and $(ii)$ and $(iii)$ in Theorem \ref{main} results in the bound
\[
\begin{split}
\sup_{s \in [0, \sigma]} \left \Vert \int_{0}^{\sigma} e^{sX}e^{tU_{\alpha}}U_{\alpha}gds \right  \Vert_{L^{2}(\mathbb{R},dt)}\\
\leq & \frac{C_{ \sigma}(\alpha)}{t^{\beta}}\max \left \{ \Vert g \Vert_{L^{2}(M)}, \Vert Xg \Vert_{L^{2}(M)}, \Vert U_{\alpha}g \Vert_{L^{2}(M)} \right \}
\end{split}
\]

\noindent where $\beta =1$ and $C_{\sigma}(\alpha)$ is a constant that depends on the time-change function $\alpha$ and parameter $\sigma > 0$.


Because $$
\int_{0}^{\sigma} e^{sX}e^{tU_{\alpha}}U_{\alpha}g\>ds
$$
\noindent is bounded on $M$, it follows that
$$
\int_{\mathbb{R}} \omega(t) \int_{0}^{\sigma} e^{sX}e^{tU_{\alpha}}U_{\alpha}g\>ds\>dt
$$
\noindent vanishes.
\end{proof}


\begin{lemma} \cite{forniulci}
Let $\omega \in L^{2}(\mathbb{R},dt)$. If for some $x \in M$ and for all $g \in C^{\infty}(M)$,
$$
\int_{\mathbb{R}} \omega(t) \int_{0}^{\sigma} e^{sX}e^{tU_{\alpha}}U_{\alpha}g(x) \> ds \> dt=0,
$$
then $\omega$ vanishes identically.
\end{lemma}

\begin{proof}

Fix $x \in M$ and $\sigma > 0$.  Since $U$ is contained in a subalgebra isomorphic to $\mathfrak{sl}_{2}$, there exists an element $V$ such that $[U,V]=2X$ and $[X,V]=-V$. For any $T >0$, $\rho >0$, and $\frac{1}{2}>\gamma > 0$, let $E^{T}_{\rho,\sigma}$ be the flow-box for the the flow $\{ \phi^{U_{\alpha}}_{t} \}$ defined as follows:
\begin{center}
$E^{T}_{\rho, \sigma}=(\phi^{U_{\alpha}}_{t} \circ \phi^{X}_{s} \circ \phi^{V}_{r})(x)$, for all $(r,s,t) \in (-\gamma, \gamma) \times (-\rho, \rho) \times (-\sigma, \sigma)$.
\end{center}

\noindent For any $\chi \in C^{\infty}_{0}(-1,1)$ and any $\psi \in C^{\infty}_{0}(-T,T)$, let
$$
\tilde{g}(r,s,t):=\chi(\frac{r}{\rho}) \chi(\frac{s}{\sigma}) \psi(t).
$$


\noindent Let $g \in C^{\infty}(M)$ such that $g=0$ on $ M\setminus Im(E^{T}_{\rho, \sigma})$ and
\begin{displaymath}
   g \circ E^{T}_{\rho, \sigma} = \left\{
     \begin{array}{lr}
       0   &  $on$   \; M\setminus Im(E^{T}_{\rho, \sigma})\\
       \tilde{g}(r,s,t)   &  $on$   \; Im(E^{T}_{\rho, \sigma})
     \end{array}
   \right.
\end{displaymath}

\noindent Let $T_{\rho, \sigma} >0$ be defined as:
$$
T_{\rho, \sigma}:= \min \left \{ |t|>T: \cup _{s \in [-\sigma,\sigma ]} (\phi^{U_{\alpha}}_{t} \circ \phi^{X}_{s})(x) \cap Im(E^{T}_{\rho, \sigma}) \neq \font\cmsy = cmsy10 at 12pt
\hbox{\cmsy \char 59} \right \}.
$$

%
From unique ergodicity,
$$
lim_{\rho \to 0^{+}} T_{\rho, \sigma} = +\infty.
$$
\noindent The composition of the flow box with $U_{\alpha}g$ and $Xg$ follow from the commutation relations:
$$
(U_{\alpha}g) \circ E^{T}_{\rho, \sigma}
: = \chi\left (\frac{r}{\rho} \right ) \chi \left (\frac{s}{\sigma} \right ) \frac{d\psi(t)}{dt}(t)
$$
\noindent and
\begin{equation} \label{1}
\begin{split}
(Xg) \circ E^{T}_{\rho, \sigma} = \frac{1}{\sigma}\chi\left (\frac{r}{\rho} \right ) \frac{d\chi}{ds}\left (\frac{s}{\sigma} \right )\psi(t)\\
 -  \left ( \int_{0}^{t}  (\frac{X\alpha}{\alpha}-1)  \circ \phi_{\tau}^{U_{\alpha}}\circ \phi_{s}^{X} \circ \phi^{V}_{r}(x) d\tau \right ) & \chi\left (\frac{r}{\rho}\right ) \chi\left (\frac{s}{\sigma}\right )\frac{d\psi}{dt}(t).
\end{split}
\end{equation}

%
%
From the assumptions of Lemma \ref{lem} and by integrating \ref{1}, we have
\begin{equation} \label{2}
\begin{split}
\chi(0)\> \left (\int_{0}^{\sigma} \chi\left (\frac{s}{\sigma}\right )\>ds \right ) \left (\int_{-T}^{T}\omega(t) \frac{d\psi(t)}{dt}\>dt \right )\\
+  \int_{\mathbb{R}\setminus[-T_{\rho, \sigma}, T_{\rho, \sigma}]} & \omega(t) \int_{0}^{\sigma} e^{sX}e^{tU_{\alpha}}(U_{\alpha}g) \> ds \> dt = 0.
\end{split}
\end{equation}

%
The bound $C_{\sigma}(\alpha)$ of
$$
\int_{0}^{\sigma} e^{sX}e^{tU_{\alpha}}U_{\alpha}gds
$$
\noindent derived for the spectral results, combined with \ref{1} and \ref{2}, give us the following $L^{2}$ bound,
%
$$
\Vert \int_{0}^{\sigma} e^{sX}e^{tU_{\alpha}}U_{\alpha}gds \Vert _{L^{2}(\mathbb{R}, dt)} \leq \frac{C_{\sigma}(\alpha)}{t}\max\{ \Vert g \Vert_{\infty}, \Vert Xg \Vert_{\infty}, \Vert U_{\alpha}g \Vert_{\infty} \}
$$
$$
\leq \frac{C_{\sigma}(\alpha)}{t}\max \left \{ 1,T \right \} \times \max \hspace{0.001cm} ^{2} \left \{ \Vert \chi \Vert_{L^{\infty}(\mathbb{R})}, \Vert \chi' \Vert_{L^{\infty}(\mathbb{R})}, \Vert \psi \Vert_{L^{\infty}(\mathbb{R})}, \Vert \psi' \Vert_{L^{\infty}(\mathbb{R})} \right \}.
$$


\noindent Since the above bound is uniform with respect to $\rho$, we can conclude that the following limit holds,
%
\begin{equation} \label{3}
\lim_{\rho \to 0^{+}} \int_{\mathbb{R}\setminus [-T_{\rho, \sigma},T_{\rho, \sigma}]}\omega(t) \int_{0}^{\sigma}  e^{sX}e^{tU_{\alpha}}(U_{\alpha}g)dsdt = 0.
\end{equation}





Combining equation \ref{2} with the limit result in \ref{3} implies that
%
$$
\int_{\mathbb{R}}\omega(t) \frac{d\psi(t)}{dt}\>dt=0
$$
%
\noindent and thus, $\omega \equiv 0$.
\end{proof}
\end{proof}


\subsection{Time changes of the horocycle flow - compact  and finite area}
\noindent On $M=\Gamma \setminus PSL(2, \mathbb{R})$, where $M$ is either compact or of finite area, we consider the basis
 \[ \left\{ U=
\begin{pmatrix}
  0 & 1 \\ 0 & 0
\end{pmatrix}, \>
V=\begin{pmatrix}
  0 & 0 \\ 1 & 0
\end{pmatrix}, \>
X=\begin{pmatrix}
  \frac{1}{2} & 0 \\ 0 & -\frac{1}{2}
\end{pmatrix} \right\}
\]
\noindent of the Lie algebra $\mathfrak{sl}_{2}(\mathbb{R})$, where $U$ and $V$ are the generators of the positive and negative horocycle flows, $\{h_{t}^{U}\}$ and $\{h_{t}^{V}\}$ respectively, and $X$ is the generator of the geodesic flow, $\{\phi_{s}^{X}\}$. From \cite{amreinnelson}, we know that $iU,iV,iX$ are essentially self-adjoint on $C^{\infty}(M)$, and thus, $U,V,X$ are essentially skew-adjoint on $C^{\infty}(M)$. It follows that time-changes of the horocycle flow are special cases of Theorem \ref{unipotent} (when $M$ is of finite volume, $\{ h_{t}^{U_{\alpha}} \}$ is ergodic, and when $M$ is compact, $\{ h_{t}^{U_{\alpha}} \}$ is uniquely ergodic).  The following Corollary was already proved in \cite{forniulci}, \cite{tiedrahor}, \cite{tiedra}
under slightly weaker regularity assumptions; in this paper we have
not attempted to optimize the regularity.

\begin{corollary}
\noindent \textbf{a.} Any smooth time-change $\{h_{t}^{U_{\alpha}}\}$ of the horocycle flow on $M$ (finite volume) has absolutely continuous spectrum on $L^{2}_{0}(M,vol_{\alpha})$ if $\Vert \frac{X\alpha}{\alpha}\Vert_{\infty}<1$.

\vspace{0.2cm}
\noindent \textbf{b.}  Any smooth time-change $\{h_{t}^{U_{\alpha}}\}$ of the horocycle flow on $M$ (compact) has Lebesgue spectrum on $L^{2}_{0}(M,vol_{\alpha})$.
\end{corollary}

\subsection{Twisted horocycle flows}

\noindent Much work has been done on the spectral analysis of skew products on tori, for example, \cite{ansai}, \cite{iwanik}, \cite{AI}, \cite{iwank}. We would like to consider a skew product for which the base dynamics are ergodic (in fact uniquely ergodic), but not an action on $S^{1}$. For such an example, we will examine the conditions under which the spectral properties persist or do not persist after we combine the horocycle time-change with a circle rotation. Our new space is $\hat{M}=(\Gamma \setminus PSL(2, \mathbb{R})) \times S^{1}$ for $\Gamma$ a cocompact lattice. We define the following operators:

\begin{center} $\hat{X} =(X,0)$ where $X$ is the generator of the geodesic flow. \end{center}
\begin{center} $\hat{V}=(V,0)$ where $V$ is the generator of the negative horocycle flow. \end{center}
\begin{center} $\hat{\frac{d}{d\theta}}=(0, \frac{d}{d\theta})$ where $\frac{d}{d\theta}$ is a rotation on $S^{1}$. \end{center}
\begin{center} $W=(U,0 )+ (0,\alpha \frac{d}{d\theta})$ where $U$ is the generator of the positive horocycle flow and $\alpha = \alpha(x)$, $x \in \Gamma \setminus PSL(2, \mathbb{R})$, is the time change function as in 3.1. \end{center}

\begin{proposition}
The flow $\{ \phi_{t}^{W}\}$ is uniquely ergodic.
\end{proposition}

\begin{proof}
Consider the time-change $\{ \phi_{t}^{W_{\alpha}}\} = \frac{1}{\alpha}W = \hat{U}_{\alpha} \times \frac{\hat{d}} {d\theta}$. Since $\{ h_{t}^{U_{\alpha}}\}$ is mixing \cite{marcusfigure}, then it is weakly mixing, and thus $\{\phi_{t}^{W_{\alpha}} \}$ is ergodic \cite{brown}. This implies the ergodicity of $\{\phi_{t}^{W}\}$. Since $\{\phi_{t}^{W}\}$ is ergodic and $\{h_{t}^{U}\}$ is uniquely ergodic \cite{furst}, then from \cite{fursten} (applied to flows), $\{\phi_{t}^{W}\}$ is uniquely ergodic.
\end{proof}

\noindent We are interested in the spectrum of the flow $\{ \phi_{t}^{W} \}$, so we compute the commutator with $\hat{X}$.
$$
e^{-tW}\left [\hat{X}, e^{tW} \right] = \boxed{tW + \left (\int_{0}^{t}(\hat{X}\alpha-\alpha)\circ \phi_{\tau}^{W}(x)\> d\tau \right )\hat{\frac{d}{d\theta}} = H(t)}
$$
\noindent For $u \in C^{\infty}(\hat{M})$,
$$
\lim_{t \to \infty} \frac{H(t)}{t} u=\boxed{\left (W - \frac{\hat{d}}{d\theta}\right )\>u = Hu}
$$
\noindent Since,
$$
\left \| \; \langle e^{tW}f,f\rangle \; \right \|_{L^{2}(\mathbb{R}, dt)}\; = \> \left \| \frac{1}{\sigma}\int_{0}^{\sigma} \langle e^{s\hat{X}}e^{tW}f, e^{s\hat{X}}f\rangle \> ds \> \right \|_{L^{2}(\mathbb{R}, dt)},
$$
\noindent the preliminary assumptions are satisfied with $B_{1}=B_{2}=I$.

%
 However, if we proceed with verifying the conditions of Theorem \ref{main} for functions in the range of $H$, we are unable to extend pointwise bounds in $L^{2}(\hat{M})$ to uniform bounds in the operator norm. Instead we modify our operators by introducing an operator $P$, defined in such a way that it not only acts as a projection operator but also preserves regularity.

%
Let $\chi \in C_{0}^{\infty}(\mathbb{R}\setminus \{0\})$ such that the support of $\chi$ is a compact subset of the spectrum of $H$ away from $0$. For $f,g \in L^{2}(\hat{M})$,
\[
\begin{split}
\langle Pf,g \rangle_{L^{2}(\hat{M})} = & \int_{\mathbb{R}} \chi (x)  d\mu_{f,g}(x)\\
 = & \int_{\mathbb{R}} \hat{\chi}(t) \hat{\mu_{f,g}} (t) \> dt\\
 = & \int_{\mathbb{R}} \hat{\chi}(t) \langle e^{tH}f,g \rangle_{L^{2}(\hat{M})} \> dt
\end{split}
\]
\noindent since $H$ is a vector field, and thus,
$$
Pf = \int_{\mathbb{R}} \hat{\chi}(t)e^{tH}f \> dt.
$$
\noindent The decay of $e^{tH}f=f \circ \phi_{t}^{H}$ is at most polynomial in $t$, however, since $\chi \in C_{0}^{\infty}(\mathbb{R}\setminus \{0\})$, $\hat{\chi} \in \mathcal{S}(\mathbb{R})$, and thus, must decay faster than any power of $\frac{1}{t}.$ In this way, we guarantee that
$$
P: C^{\infty}(\hat{M}) \to C^{\infty}(\hat{M}),
$$
\noindent and we take
$$
D = P(C^{\infty}(\hat{M})).
$$
\noindent Now we introduce our modified operators.

%
Let
$$
\hat{X}_{p}= P\hat{X}P.
$$
\noindent Since $P$ commutes with $e^{tH}$, $P$ commutes with $H$. Thus, $P$ commutes with $W$ and $e^{tW}$. Therefore,
$$
e^{-tW}\left [\hat{X}_{P}, e^{tW} \right]=Pe^{-tW}\left [\hat{X},e^{tW} \right]P= PH(t)P=H_{P}(t).
$$
\noindent For $u \in C^{\infty}(\hat{M})$,
$$
\lim_{t \to \infty} \frac{H_{P}(t)}{t} u=PHPu=HP^{2}u= H_{P}u.
$$
\noindent Note that now $H_{P}$ is a bounded, invertible operator. Let
$$
C_{H_{P}}=\left \| H_{P} \right \|_{op}
$$
$$
C_{H_{P}}^{-1}=\| H_{P}^{-1} \|_{op}
$$
$$
C_{P}^{k} = \| P^{k} \|_{op}
$$
$$
C_{H}^{P} = \| HP \|_{op}
$$
$$
C_{\alpha}'=\| \hat{X}\alpha - \alpha \|_{\infty}
$$

\begin{theorem} \label{twisted}
The flow $\{\phi_{t}^{W}\}$ has absolutely continuous spectrum on $\overline{Ran(H)}$.
\end{theorem}

\begin{proof}
\noindent We will verify the conditions of Theorem \ref{main} on each subspace
$$
\mathbb{E}_{n} = \left \{\hat{\frac{d}{d \theta}}u=inu: u \in L^{2}(\hat{M}) \right \}.
$$
\noindent $(i)$
\[
\begin{split}
\frac{H_{P}(t)}{t} =& PWP + Pin\left (\frac{1}{t}\int_{0}^{t}(\hat{X}\alpha - \alpha)\circ \phi_{\tau}^{W}(x) \> d\tau \right)P\\
=& PHP + PinP + Pin \left (\frac{1}{t}\int_{0}^{t}(\hat{X}\alpha - \alpha)\circ \phi_{\tau}^{W}(x) \> d\tau \right )P\\
=& H_{P} + Pin\left( \frac{L(t)}{t}+1 \right )P
\end{split}
\]
\noindent for
$$
L(t) = \int_{0}^{t}(\hat{X}\alpha - \alpha)\circ \phi_{\tau}^{W}(x) \> d\tau.
$$
\noindent Let $f=H_{P}g$,
\[
\begin{split}
\left \Vert \frac{H_{P}(t)}{t}H_{P}^{-1}f \right \Vert_{L^{2}(\hat{M})} = & \left \Vert \frac{H_{P}(t)}{t}g \right \Vert_{L^{2}(\hat{M})}\\
= & \left \Vert H_{P}g + inP\left (\frac{L(t)}{t}+1 \right )P g \right \Vert_{L^{2}(\hat{M})}\\
\leq &  \left \Vert H_{P}g  \right \Vert_{L^{2}(\hat{M})} + \left \Vert inP\left (\frac{L(t)}{t}+1 \right )P g \right \Vert_{L^{2}(\hat{M})} \\
\leq & C_{H_{P}}\big \Vert g \big \Vert_{L^{2}(\hat{M})} + \>  nC_{P}^{2}\left \Vert \hat{X}\alpha -\alpha \right \Vert_{\infty} \left \Vert g \right \Vert_{L^{2}(\hat{M})}\\
+ &  \> n C_{P}^{2} \left \Vert g \right \Vert_{L^{2}(\hat{M})}\\
= & \left (C_{H_{P}} + nC_{P}^{2}(C_{\alpha}'+1) \right)  \left \Vert g \right \Vert_{L^{2}(\hat{M})}.
\end{split}
\]
\noindent So,
$$
\left \Vert \frac{H_{P}(t)}{t}H_{P}^{-1} \right \Vert_{op} \leq C_{H_{P}} + nC_{P}^{2}(C_{\alpha}'+1).
$$
\noindent Also,
\[
\begin{split}
\left \| \left (I - \frac{H_{p}(t)}{t}H_{p}^{-1}\right )f \right \|_{L^{2}(\hat{M})} = & \left \| \left (I - I - inP\left (\frac{L(t)}{t}+1 \right )PH_{p}^{-1} \right )f \right \|_{L^{2}(\hat{M})}\\
= & \left \| inP\left ( \frac{L(t)}{t}+ 1 \right )Pg \right \|_{L^{2}(\hat{M})}\\
\leq & \>  n \> C_{P}^{2} \left \| \frac{L(t)}{t} + 1 \right \|_{\infty}\cdot\| g \|_{L^{2}(\hat{M})}.
\end{split}
\]
\noindent Since $\{ \phi_{t}^{W}\}$ is uniquely ergodic, the following converges uniformly,
$$
\lim_{t \to \infty} \frac{L(t)}{t} + 1 = \lim_{t \to \infty}\frac{1}{t}\int_{0}^{t}(\hat{X}\alpha - \alpha) \circ \phi_{\tau}^{W}(x) \>d\tau + 1 =  0.
$$
\noindent So,
$$
\limsup_{t \to \infty} \left \Vert I - \frac{H_{p}(t)}{t}H_{p}^{-1} \right \Vert_{op} \leq \limsup_{t \to \infty} n(C_{P})^{2} \left \Vert \left ( \frac{L(t)}{t} +1 \right ) \right \Vert_{\infty}=0,
$$
\noindent and hence,
$$
\limsup_{t \to \infty} \left \| I - \frac{H_{p}(t)}{t}H_{p}^{-1} \right \|_{op}<1.
$$
\noindent $(ii)$
\[
\begin{split}
\left [\hat{X}_{P}, \frac{H_{p}(t)}{t}(H_{P})^{-1} \right ]
= & \underbrace{P \left [\hat{X},P \right ]P\frac{H(t)}{t}PH_{P}^{-1}}_{a}\\
+ & \underbrace{P^{2} \left [\hat{X},P \right]\frac{H(t)}{t}PH_{P}^{-1}}_{b}\\
+ & \underbrace{P^{3}\left [\hat{X},\frac{H(t)}{t}\right ]PH_{P}^{-1}}_{c}\\
+ & \underbrace{P^{3}\frac{H(t)}{t}\left [\hat{X},P \right]H_{P}^{-1}}_{d}\\
+ & \underbrace{P^{3}\frac{H(t)}{t}P\left [\hat{X},H_{P}^{-1} \right]}_{e}.
\end{split}
\]
\noindent Before we bound terms a-e, we show bounds for the terms $\left [\hat{X},P \right ]$ and $\left [\hat{X}, \frac{H(t)}{t} \right]P$.
\[
\begin{split}
\left [\hat{X},P \right]f = & \int_{\mathbb{R}} \hat{\chi}(t) \left[\hat{X},e^{tH} \right ]f \> dt\\
= & \int_{\mathbb{R}} \hat{\chi}(t)e^{tW}e^{-tW}\left [\hat{X},e^{tH+tin-tin} \right]f \> dt\\
= & \int_{\mathbb{R}} \hat{\chi}(t)e^{tW}e^{-tW}\left [\hat{X},e^{tW} \right]e^{-itn}f \> dt\\
= & \int_{\mathbb{R}} \hat{\chi}(t)e^{tW}H(t)e^{-itn}f \> dt \\
= & \int_{\mathbb{R}} \hat{\chi}(t)e^{tW}tWe^{-itn}f \> dt +  \int_{\mathbb{R}} \hat{\chi}(t)e^{tW}(in L(t))e^{-itn}f \> dt\\
= & \int_{\mathbb{R}} \hat{\chi}(t)e^{-itn}t\frac{d}{dt}(e^{tW}f) \> dt +  \int_{\mathbb{R}} \hat{\chi}(t)e^{tW}(in L(t))e^{-itn}f \> dt.
\end{split}
\]
\noindent The first term we integrate by parts:
$$
\hspace{1cm} \hat{\chi}(t)e^{-itn}t \hspace{1cm} \hat{\chi}^{'}(t)e^{-itn}t + \hat{\chi}(t)e^{-itn}-in\hat{\chi}(t)e^{-itn}t
$$
$$
\hspace{-2.5cm} e^{tW}f \hspace{1cm} \frac{d}{dt}(e^{tW}f)
$$
%
\[
\begin{split}
\int_{\mathbb{R}} \hat{\chi}(t)e^{-itn}t\frac{d}{dt}(e^{tW}f) \> dt
 = & \>  \hat{\chi}(t)e^{-itn}te^{tW}f\biggr\rvert_{\infty}^{-\infty}\\
 + &  \int_{\mathbb{R}} \left ( \hat{\chi}^{'}(t)e^{-itn}t + \hat{\chi}(t)e^{-itn}-in\hat{\chi}(t)e^{-itn}t \right)e^{tW}f  \> dt\\
= & \int_{\mathbb{R}} \left ( \hat{\chi}^{'}(t)e^{-itn}t + \hat{\chi}(t)e^{-itn}-in\hat{\chi}(t)e^{-itn}t \right)e^{tW}f  \> dt.
\end{split}
\]
\noindent So,
\[
\begin{split}
& \left \| \int_{\mathbb{R}} \left ( \hat{\chi}^{'}(t)e^{-itn}t + \hat{\chi}(t)e^{-itn}-in\hat{\chi}(t)e^{-itn}t \right)e^{tW}f  \> dt \right \|_{L^{2}(\hat{M})}\\
\leq & \left (\int_{\mathbb{R}} \left | \hat{\chi}^{'}(t)t \right | \> dt + \int_{\mathbb{R}} \left |\hat{\chi}(t) \right| \> dt + \int_{\mathbb{R}} \left |n\hat{\chi}(t) \right |  \> dt \right)  \left \| f \right \|_{L^{2}(\hat{M})}\\
\leq & C_{1} \left \| f \right \|_{L^{2}(\hat{M})} .
\end{split}
\]
\noindent The boundedness of the second term follows immediately,
\[
\begin{split}
\left \| \int_{\mathbb{R}} \hat{\chi}(t)e^{tW}(in L(t))e^{-itn}f  \> dt \right \|_{L^{2}(\hat{M})}
\leq & \int_{\mathbb{R}} \left| \hat{\chi}(t)(in L(t)) \right| \> dt \| f \|_{L^{2}(\hat{M})}\\
\leq & C_{2}  \| f \|_{L^{2}(\hat{M})}.
\end{split}
\]
\noindent Thus,
$$
\left \| \left [\hat{X},P \right ] f \right \|_{L^{2}(\hat{M})} \leq (C_{1}+C_{2}) \| f \|_{L^{2}(\hat{M})} = C \| f \|_{L^{2}(\hat{M})},
$$
\noindent and hence,
$$
\left \| \left [\hat{X},P \right ] \right \|_{op}\leq C.
$$
\noindent Also,
\[
\begin{split}
\left [\hat{X},\frac{H(t)}{t} \right ]P = & \left [\hat{X},W \right]P + \left [\hat{X},\frac{1}{t}\int_{0}^{t}(\hat{X}\alpha - \alpha) \circ \phi_{\tau}^{W}(x) \> d\tau \> in \right]P\\
= & \> (\hat{U} + \hat{X}\alpha \> in)P + [\hat{X},\frac{1}{t}\int_{0}^{t}(\hat{X}\alpha - \alpha) \circ \phi_{\tau}^{W}(x) \> d\tau \> in]P\\
= & \left (\hat{U} + (\alpha-1)in -(\alpha -1)in + \hat{X}\alpha \> in \right )P\\
+ & \left [\hat{X},\frac{1}{t}\int_{0}^{t}(\hat{X}\alpha - \alpha) \circ \phi_{\tau}^{W}(x) \> d\tau \> in \right]P\\
= & \> HP + \left (\hat{X}\alpha - \alpha +1 \right)\> inP + \left [\hat{X},\frac{1}{t}\int_{0}^{t}(\hat{X}\alpha - \alpha) \circ \phi_{\tau}^{W}(x) \> d\tau \> in \right]P.
\end{split}
\]
\noindent Now we bound the following term,
\[
\begin{split}
& \left [\hat{X},\frac{1}{t}\int_{0}^{t}(\hat{X}\alpha - \alpha) \circ \phi_{\tau}^{W}(x) \> d\tau \> in \right] \\
= &  \> \hat{X} \left (\frac{1}{t}\int_{0}^{t}(\hat{X}\alpha - \alpha)\circ \phi_{\tau}^{W}(x) \>d\tau \right )in\\
- &  \left (\frac{1}{t}\int_{0}^{t}(\hat{X}\alpha - \alpha)\phi_{\tau}^{W}(x) \>d\tau \right)\hat{X}in\\
= & \> \frac{1}{t}\int_{0}^{t}\left (D\phi_{\tau}^{W}(\hat{X})\circ \phi_{-\tau}^{W}(x) \right )(\hat{X}\alpha - \alpha)\circ \phi_{\tau}^{W}(x) \>d\tau \>  in.
\end{split}
\]
\noindent Since $\hat{X}\alpha - \alpha$ is a function on $M$,
\[
\begin{split}
\frac{1}{t}\int_{0}^{t}\left (D\phi_{\tau}^{W}(\hat{X})\circ \phi_{-\tau}^{W}(x) \right )&(\hat{X}\alpha - \alpha)\circ \phi_{\tau}^{W}(x) \>d\tau \>  in\\
= & \> \frac{1}{t}\int_{0}^{t} \left(D\phi_{\tau}^{U}(X)\circ \phi_{-\tau}^{U}(x) \right) (X\alpha - \alpha)\circ \phi_{\tau}^{U}(x) \>d\tau \>  in
\end{split}
\]\[
\begin{split}= & \> \frac{1}{t}\int_{0}^{t}(X+ \tau U)(X\alpha -\alpha)\circ \phi_{\tau}^{U}(x) \>d\tau \>  in.
\end{split}
\]
\noindent We consider the $L^{\infty}(M)$ norm and let
$$
\left \|
 X\alpha \right \|_{\infty} = C_{\alpha,M} < \infty
$$
$$
\left \|
 X\alpha - \alpha \right \|_{\infty} = C_{\alpha,M}^{'} < \infty
 $$
\noindent and
$$
\left \| X(X \alpha ) \right \|_{\infty}= C_{\alpha,M}'' < \infty.
$$
\noindent Hence,
\[
\begin{split}
& \left \| \frac{1}{t}\int_{0}^{t}(X+ \tau U)(X\alpha - \alpha)\circ \phi_{\tau}^{U}(x) \>d\tau \>  in \right \|_{\infty}\\
\leq & \left \| \frac{1}{t}\int_{0}^{t}X(X\alpha - \alpha)\circ \phi_{\tau}^{U}(x) \>d\tau \>  in \|_{\infty} + \| \frac{1}{t}\int_{0}^{t}\tau U(X\alpha - \alpha)\circ \phi_{\tau}^{U}(x) \>d\tau \>  in \right  \|_{\infty}\\
\leq & \> n \| X(X\alpha - \alpha) \|_{\infty}  \> + \> n \| (X\alpha - \alpha) \circ \phi_{t}^{U}-(X\alpha - \alpha)  \|_{\infty} \\
\leq & \> n \left \| X(X\alpha - \alpha) \|_{\infty}   + 2n \| X\alpha - \alpha  \right \|_{\infty} \\
\leq &  \> n \| X(X\alpha) \|_{\infty}   + \> n \| X \alpha \|_{\infty}  + 2n \| X\alpha - \alpha  \|_{\infty} \\
\leq & \> nC_{\alpha,M}'' \>  + n C_{\alpha,M} +  2nC_{\alpha,M}^{'} .
\end{split}
\]
\noindent Thus, $\left [\hat{X}, \frac{H(t)}{t} \right ]P$ extends to a bounded operator on $\overline{Ran(H_{p})}$ with operator norm uniformly bounded in $t$:
$$
\left \| \left [\hat{X}, \frac{H(t)}{t} \right ]P \right \|_{op} \leq \left (C_{H}^{P}+n(C_{\alpha}^{'}+1)C_{P}^{1} \right ) + n\left (C_{\alpha,M}'' \>  + C_{\alpha,M} +  2C_{\alpha,M}^{'} \right)C_{P}^{1}.
$$
\noindent $a$:
\[
\begin{split}
\left \| P \left [\hat{X},P \right ]P\frac{H(t)}{t}PH_{P}^{-1} \right \|_{op}\leq & \> C^{1}_{P}\cdot C \cdot \left \| \frac{H_{P}(t)}{t}H_{P}^{-1} \right \|_{op}\\
\leq & \> C^{1}_{P}\cdot C \cdot \left (C_{H_{P}} + nC_{P}^{2}(C_{\alpha}'+1) \right).
\end{split}
\]
\noindent $b:$
\[
\begin{split}
\left \| P^{2} \left [\hat{X},P \right ]\frac{H(t)}{t}PH_{P}^{-1} \right \|_{op}= & \left \| P^{2} \left [\hat{X},P \right ](H + \left (\frac{L(t)}{t} + 1 \right ) in)PH_{P}^{-1} \right \|_{op}\\
\leq & \> C_{P}^{2} \cdot C \cdot \left ( \left \| HP \right \|_{op} +\> n\left ( \left \| \hat{X}\alpha - \alpha \right \|_{\infty} + 1 \right ) C_{P}^{1} \right )\\   \cdot & \| H_{P}^{-1} \|_{op}\\
= & \> C_{P}^{2} \cdot C \cdot \left (C_{H}^{P} + n(C_{\alpha}'+1)C_{P}^{1} \right )C_{H_{P}}^{-1}
\end{split}
\]
\noindent $c:$
\[
\left \| P^{3} \left [\hat{X},\frac{H(t)}{t} \right ]PH_{P}^{-1} \right \|_{op}
\]
\[
\leq C_{P}^{3}\left ((C_{H}^{P}+n \left (C_{\alpha}^{'}+1)C_{P}^{1} \right ) + n(C_{\alpha,M}'' \>  + C_{\alpha,M} +  2C_{\alpha,M}^{'})C_{P}^{1} \right ) C_{H_{P}}^{-1}.
\]
\noindent $d:$
\[
\left \| P^{3}\frac{H(t)}{t}\left [\hat{X},P \right ]H_{P}^{-1} \right \|_{op} \leq C_{P}^{2}\cdot \left (C_{H}^{P} + nC_{P}^{1}(C_{\alpha}'+1) \right )\cdot C \cdot C_{H_{P}}^{-1}
\]
\noindent $e:$
\[
\begin{split}
 \left [\hat{X},H_{P}^{-1} \right ]= & H_{P}^{-1}\left [H_{P},\hat{X} \right ]H_{P}^{-1}\\
= & H_{P}^{-1}\left [P,\hat{X} \right ]PHH_{P}^{-1} + H_{P}^{-1}P\left [P,\hat{X} \right ]HH_{P}^{-1} + H_{P}^{-1}P^{2}\left [\hat{X},H \right]H_{P}^{-1}.
\end{split}
\]
\[
\begin{split}
\left \| \left [\hat{X},H_{P}^{-1} \right ] \right \|_{op}
 \leq & \> C_{H_{P}}^{-1}\cdot C \cdot C_{H}^{P}C_{H_{P}}^{-1} + C_{H_{P}}^{-1}C_{P}^{1}\cdot C \cdot C_{H}^{P}C_{P}^{-1}C_{H_{P}}^{-1}\\
 + & \> C_{H_{P}}^{-1}(C_{H_{P}}+C_{P}^{2}n(C_{\alpha}^{'}+1))C_{H_{P}}^{-1}.
\end{split}
\]
\[
\begin{split}
& \left \| P^{3}\frac{H(t)}{t}P\left [\hat{X},H_{P}^{-1}\right ] \right \|_{op}
\leq  C_{P}^{2}\left (C_{H_{P}}+nC_{P}^{2}(C_{\alpha}'+1) \right )\\
\cdot & \left ( C_{H_{P}}^{-1}\cdot C \cdot C_{H}^{P}C_{H_{P}}^{-1} + C_{H_{P}}^{-1}C_{P}^{1}\cdot C \cdot C_{H}^{P}C_{P}^{-1}C_{H_{P}}^{-1} + C_{H_{P}}^{-1}(C_{H_{P}}+C_{P}^{2}n(C_{\alpha}^{'}+1)C_{H_{P}}^{-1} \right ).
\end{split}
\]


\noindent $(iii)$.
\[
\begin{split}
\left [ H_{P}(t),H_{P} \right  ] H_{P}^{-1} = & \left [tH_{P} + P(L(t)+t)Pin, H_{P} \right ]H_{P}^{-1}\\
= & \left [PL(t)Pin,H_{P} \right ]H_{P}^{-1}\\
= & P^{3}HL(t)PinH_{P}^{-1}=P^{3}WL(t)PinH_{P}^{-1}-P^{3}inL(t)PinH_{P}^{-1}\\
= & P^{3}\left ((\hat{X}\alpha-\alpha)\circ \phi_{t}^{W} - (\hat{X}\alpha - \alpha)\right )PinH_{P}^{-1}-P^{3}inL(t)PinH_{P}^{-1}.
\end{split}
\]
\[
\begin{split}
\left \| \left [H_{p}(t),H_{P} \right ]H_{P}^{-1} \right \|_{op}\\
\leq & \> 2n \> C_{P}^{4}C_{\alpha}'C_{H_{P}}^{-1}+n^{2}C_{P}^{4}C_{\alpha}^{'}C_{H_{P}}^{-1}.
\end{split}
\]
\noindent We have shown that the conditions of Theorem \ref{main} are satisfied on $Ran(P)$. We would like to extend this to $Ran(H)$. Recall that $H_{P}$ depends upon a choice of $\chi \in C_{0}^{\infty}(\mathbb{R}\setminus \{0\})$. For $f \in Dom(H)$, we can express the following in terms of integrals involving the spectral projector as
$$
Hf = \int_{\mathbb{R}} x \> dE(x)f,
$$
$$
H_{P}f = \int_{\mathbb{R}}x\chi(x) \> dE(x)f,
$$
\noindent and since $f \in Dom(H)$,
$$
\int_{\mathbb{R}} x^{2} \> dE(x)f< + \infty.
$$
\noindent Let $\chi$ be such that \begin{center}
$\chi(x) = 1$ \hspace{0.5cm} for $x \in (-K,-\epsilon)\cup (\epsilon, K)=I_{\epsilon,K}$
\end{center}
\noindent and $supp(\chi)$ vanishes outside of $I_{\epsilon,K}$. Since on $I_{\epsilon,K}$,
$$
H_{P}f=Hf,
$$
\noindent we consider
$$
H_{P}f - Hf
$$
\noindent on $\mathbb{R}\setminus I_{\epsilon,K}$, i.e.,
$$
\int_{\mathbb{R}\setminus I_{\epsilon,K}}x(\chi(x)-1)\>dE(x)f.
$$
\noindent For $|x|\leq \epsilon$,
\[
\begin{split}
\lim_{\epsilon \to 0} \left \| \int_{|x|\leq \epsilon} x\left (\chi(x)-1 \right )\>dE(x)f \>  \right \|^{2}_{L^{2}(\mathbb{R})} = & \lim_{\epsilon \to 0}  \int_{|x|\leq \epsilon} \left |x(\chi(x)-1) \right |^{2}\>dE(x)f\\
\leq & \lim_{\epsilon \to 0}  4\epsilon^{2} \int_{|x|\leq \epsilon}\>dE(x)f\\
\leq & \lim_{\epsilon \to 0}   4\epsilon^{2} \| f \|_{L^{2}(\hat{M})}\\
= & 0.
\end{split}
\]
\noindent For $|x| \geq K$,
\[
\begin{split}
\lim_{K \to \infty} \left \| \int_{|x| \geq K} x(\chi(x)-1)\>dE(x)f \right \|^{2}_{L^{2}(\mathbb{R})} =  & \lim_{K \to \infty}  \int_{|x| \geq K} \left | x(\chi(x)-1) \right |^{2}\>dE(x)f\\
\leq & \lim_{K \to \infty}  \int_{|x| \geq K} 4x^{2}\>dE(x)f =0
\end{split}
\]
\noindent since
$$
\int_{\mathbb{R}} x^{2} \>dE(x)f< + \infty.
$$
\noindent Thus,
$$
\inf_{\chi \in C^{\infty}_{0}(\mathbb{R} \setminus \{0\})} \left \| H_{P}f-Hf \right \|_{L^{2}(\hat{M})} =0.
$$
\noindent So, for any $Hf$, there exists a sequence $\{H_{P_{n}}\}$ such that
$$
H_{P_{n}} \to Hf,
$$
\noindent and thus,
$$
\overline{Ran(H)} = \overline{\left \{ \bigcup_{P}Ran(H_{P}) \right \}}.
$$
\noindent Consequently, for every $f \in \overline{Ran(H)}$, $\mu_{f}$ is absolutely continuous.
\end{proof}
\noindent The characteristics of $\overline{Ran(H)}$ are linked to the properties of the cocycle
$$
a(x,t)=\int_{0}^{t}(\alpha -1) \circ h^{U}_{s}(x) \> ds
$$
\noindent since
$$
\phi_{t}^{H}(x,\theta)=\left (h^{U}_{t}(x), \theta + \int_{0}^{t}(\alpha -1) \circ h^{U}_{s}(x) \> ds \right).
$$
\noindent For $\pi$ a projection from $\mathbb{R}$ to $S^{1}$, let
$$
\hat{a} = \pi(a).
$$


\noindent Let
$$
k_{0}=\min \left \{k: k\hat{a}(x,t) = g\circ h_{t}^{U}(x)-g(x) \right \},
$$

\noindent for $k \in \mathbb{Z}^{+}$ and $g: M \to S^{1}$.

%
From \cite{ansai} we have the following cases.

%
If $k_{0}=\infty$ then $\{ \phi_{t}^{W} \}$ has purely absolutely continuous spectrum on $L^{2}_{0}(\hat{M})$; this is the case when $\{ \phi_{t}^{H} \}$ is ergodic, and thus,
$$
\overline{Ran(H)}=L^{2}_{0}(\hat{M}).
$$

%
When, $k_{0}<\infty$, we have a nontrivial pure point component of the spectrum. To show this, we consider the subspaces given by
$$
E_{nk_{0}}=\left \{f(x)e^{ink_{0}\theta} \right \} \subset L^{2}(\hat{M})
$$
\noindent for $f \in L^{2}(M)$ and $n \in \mathbb{Z}$.

%
The operator
$$
H_{t}:f(x) \to f \circ h_{t}^{U}(x)\>e^{ink_{0}\hat{a}(x,t)}
$$
\noindent is unitarily equivalent to the restriction of the unitary group $\{\phi_{t}^{H}\}$ to $E_{nk_{0}}$.

%
Let
$$
S_{t}:f(x) \to f \circ h_{t}^{U}(x)
$$

\noindent and
$$
V_{t}:f(x) \to f(x)\>e^{-ing(x)}.
$$

%
We compute $H_{t} \circ V_{t}$:
\[
\begin{split}
\left (H_{t} \circ V_{t} \right )(f)(x)
= & f \circ h_{t}^{U}(x)\>e^{-ing \circ h_{t}^{U}(x) }\>e^{ink_{0}\hat{a}(x,t)}\\
= & f \circ h_{t}^{U}(x)\>e^{-ing \circ h_{t}^{U}(x)}\>e^{ing \circ h_{t}^{U}(x)-ing(x)}\\
= & f \circ h_{t}^{U}(x)\>e^{-ing(x)} \\
= & \left (V_{t} \circ S_{t} \right )(f)(x).
\end{split}
\]
\noindent Hence, on $E_{nk_{0}}$, $H_{t}$ is unitarily equivalent to $S_{t}$. Since $e^{ink_{0}\theta}$ is an invariant function for $S_{t}$, $H_{t}$ has an eigenfunction in $E_{nk_{0}}$ for every $n$. Thus, the spectrum on $E_{nk_{0}}$ has  an absolutely continuous component as well as an infinite dimensional pure point component. This leads to the following Corollary.

\begin{corollary} \label{moretwisted}
The spectrum of $\{ \phi_{t}^{W} \}$ on $C^\perp$, the orthocomplement of the constants, has a pure point component (possibly trivial) and an absolutely continuous component, but no singularly continuous component.
\end{corollary}

\section{Applications to maps}

\noindent The author in \cite{tiedra} uses the Mourre Estimate \cite{amrein}, \cite{mourre} to prove the following spectral results. Here we rederive similar results by showing that the conditions of Theorem \ref{main} are satisfied; our results are slightly weaker as we prove results for functions of class $C^{2}$ while the author in \cite{tiedra} considers functions of class $C^{1}$ with an added Dini Condition. We use the notation and description from \cite{tiedra}.

\subsection{Skew products over translations}

%
Let $X$ be a compact metric abelian Lie group with normalized Haar measure $\mu$. Let $\{F_{t}\}$ be a uniquely ergodic \cite{fursten} translation flow (we assume that $F_{1}$ is ergodic), \begin{center} $F_{t} = y_{t}x$ with vector field $Y$. \end{center} The associated operators $\{V_{t}\}$ are given by \begin{center} $V_{t}\psi=\psi \circ F_{t}$ with generator $P=-i\mathscr{L}_{Y}$. \end{center} Let $G$ be a compact metric abelian group. Let $\phi:$ $X \to G$ such that $\phi$ can be written as $\phi=\xi \eta$ where $\xi$ is a group homomorphism and $\eta$ satisfies
$$
\sup_{t > 0} \left \Vert \frac{\mathscr{L}_{Y}(\chi \circ \eta) \circ F_{t} - \mathscr{L}_{Y}(\chi \circ \eta)}{t}) \right \Vert_{\infty} < \infty
$$
\noindent and
$$
\chi \circ \eta = e^{i\tilde{\eta}_{\chi}}
$$
\noindent for $\tilde{\eta}_{\chi} \in Dom(P)$ a real-valued function determined by $\chi$ and $\eta$.

%
%
The skew product, $T:X \times G \to X \times G$, is defined by
$$
T(x,z)=(y_{1}x,\phi(x) z)
$$
\noindent with corresponding unitary operator \begin{center} $W \Psi = \Psi \circ T$. \end{center} Let $\hat{G}$ be the character group of $G$. The decomposition $L^{2}(X \times G)=\bigoplus_{\chi \in \hat{G}}L_{\chi}$, $L_{\chi} = \{ \varphi \otimes \chi : \varphi \in L^{2}(X)\}$ and the restriction of $W$ to the subspaces $L_{\chi}$ allow us to study the spectrum of convenient, unitarily equivalent operators to $W|_{L_{\chi}}$, namely,
\[
U_{\chi} \psi = (\chi \circ \psi) V_{1} \psi
\]

\noindent for $\chi \circ \xi \not\equiv 1$. (Here $U_{\chi}$ takes the place of $e^{U}$ as given in the conditions.)

%
We will choose to take the commutator with $P$; it follows from \cite{amreinnelson} that $P$ is essentially self-adjoint  on $D = C^{\infty}(X)$.
\[
\begin{split}
\left [P, U_{\chi}\right ]= &\left [P,(\chi \circ \phi)V_{1}\right ]\\
= & \left [P, (\chi \circ \phi)I \right ]V_{1}\\
= & -i\mathscr{L}_{Y}\left (\chi \circ \phi \right )V_{1}\\
= & -i\left [\mathscr{L}_{Y}(\chi \circ \xi)(\chi \circ \eta) + (\chi \circ \xi)\mathscr{L}_{Y}(\chi \circ \eta) \right ]V_{1}\\
= & -i\left [\xi_{0} + \frac{\mathscr{L}_{Y}(\chi \circ \eta)}{(\chi \circ \eta)} \right ](\chi \circ \phi)V_{1}\\
= & -i\left [\xi_{0} + \frac{\mathscr{L}_{Y}(\chi \circ \eta)}{(\chi \circ \eta)} \right ]U_{\chi}
\end{split}
\]
\noindent where $\xi_{0}=\frac{d}{dt}(\chi \circ \xi)(y_{t})|_{t=0} \in i\mathbb{R}\setminus\{0\}$.

%
%
%
%
So,
$$
[P, U_{\chi}] =\left (- i\xi_{0}-\frac{i\mathscr{L}_{Y}(\chi \circ \eta)}{(\chi \circ \eta)} \right )U_{\chi} = GU_{\chi}.
$$
\noindent Thus,
\[
U_{\chi}^{-n}[P, U_{\chi}^{n}] = \sum_{k=1}^{n}U_{\chi}^{-k}GU_{\chi}^{k}
\]
\[
=\boxed{\sum_{k=1}^{n}\left (- i\xi_{0}-\frac{i\mathscr{L}_{Y}(\chi \circ \eta)}{(\chi \circ \eta)}\right )\circ F_{-k}=\sum_{k=1}^{n}G\circ F_{-k}= H(n).}
\]
\noindent Note that
\[
\begin{split}
\frac{-i\mathscr{L}_{Y}(\chi \circ \eta)}{(\chi \circ \eta)} = & \frac{-i\mathscr{L}_{Y}(e^{i\tilde{\eta}_{\chi}})}{e^{i\tilde{\eta}_{\chi}}}\\
= & \frac{-ie^{i\tilde{\eta}_{\chi}} \cdot \mathscr{L}_{Y}(\tilde{\eta}_{\chi})}{e^{i\tilde{\eta}_{\chi}}}\\
= & -i\mathscr{L}_{Y}(\tilde{\eta}_{\chi})\\
= & \> P\tilde{\eta}_{\chi}.
\end{split}
\]

%
%
From unique ergodicity we get the following convergence
\[
\lim_{n \to \infty} \frac{H(n)}{n}u =\left (-i\xi_{0} + \int_{X} P \tilde{\eta}_{\chi} \> d\mu \right )u =\boxed{-i\xi_{0}u=Hu}
\]
\noindent uniformly in $n$ for $u \in L_{\chi}$.

%
Since
\[
\left \Vert \; \langle U_{\chi}^{n}f,f \rangle  _{L_{\chi}}\; \right \Vert_{\ell^{2}(\mathbb{Z})} \; =\; \left \Vert  \frac{1}{\sigma}\int_{0}^{\sigma} \langle e^{sP}U_{\chi}^{n}f, e^{sP}f \rangle_{L_{\chi}} ds\> \right \Vert_{\ell^{2}(\mathbb{Z})},
\]
\noindent the preliminary assumptions for Theorem \ref{main} are satisfied for $B_{1}=B_{2}=I$. Now we proceed by showing that the conditions of Theorem \ref{main} hold.


\noindent  $(i)$ It is unnecessary to consider coboundaries since both $H=-i\xi_{0}I$ and $H^{-1}=\frac{i}{\xi_{0}}I$ are constants. Instead we take any $f \in L_{\chi}$.
\[
\begin{split}
\frac{H(n)}{n}H^{-1}f = & \left (\frac{1}{n}\sum_{k=1}^{n}\left (- i\xi_{0}-\frac{i\mathscr{L}_{Y}(\chi \circ \eta)}{(\chi \circ \eta)} \right )\circ F_{-k} \right ) \cdot \frac{i}{\xi_{0}}f\\
= & f + \left (\frac{1}{n}\sum_{k=1}^{n}  P\tilde{\eta}_{\chi} \circ F_{-k} \right ) \cdot \frac{i}{\xi_{0}}f .
\end{split}
\]
\noindent So,
$$
\left \Vert \frac{H(n)}{n}H^{-1}f \right \Vert_{L_{\chi}} \leq \left (1 + \frac{\Vert P\tilde{\eta}_{\chi} \Vert_{L_{\chi}}}{|\xi_{0}|} \right ) \cdot \Vert f \Vert_{L_{\chi}}.
$$
\noindent Since $\tilde{\eta}_{\chi} \in Dom(P)$,
$$
\Vert P\tilde{\eta}_{\chi} \Vert_{L_{\chi}} \leq C_{1}.
$$

%
%
 Thus, $\frac{H(n)}{n}H^{-1}$ is a bounded operator with uniformly bounded norm in $n$,
$$
\left \Vert \frac{H(n)H^{-1}}{n} \right \Vert_{op} \leq 1 + \frac{C_{1}}{|\xi_{0}|}.
$$

\noindent Also,
\[
\begin{split}
\left \Vert \left (I - \frac{H(n)}{n}H^{-1} \right ) f \right \Vert_{L_{\chi}}= &
\left \Vert \left (1- \left (1+ \left (\frac{1}{n}\sum_{k=1}^{n}  P\tilde{\eta}_{\chi} \circ F_{-k}\right ) \cdot \frac{i}{\xi_{0}}\right )\right )f \right  \Vert_{L_{\chi}}\\
= &\left \Vert \left (\frac{1}{n}\sum_{k=1}^{n}  P\tilde{\eta}_{\chi} \circ F_{-k} \cdot \frac{i}{\xi_{0}} \right )f \right \Vert_{L_{\chi}}\\
\leq & \left \Vert \left (\frac{1}{n}\sum_{k=1}^{n}  P\tilde{\eta}_{\chi} \circ F_{-k} \right ) \cdot \frac{i}{\xi_{0}} \right \Vert_{\infty} \cdot \Vert f \Vert_{L_{\chi}}.
\end{split}
\]
\noindent As a result of unique ergodicity, the following converges uniformly,
$$
\lim_{n \to \infty} \left (\frac{1}{n}\sum_{k=1}^{n}  P\tilde{\eta}_{\chi} \circ F_{-k} \right ) \cdot \frac{i}{\xi_{0}}= \frac{i}{\xi_{0}}\int_{X} P \tilde{\eta}_{\chi} \> d\mu=0,
$$
\noindent and thus,
\[
\begin{split}
\limsup_{n \to \infty} \left \Vert I - \frac{H(n)}{n}H^{-1}  \right \Vert_{op}\leq & \limsup_{n \to \infty} \left \Vert \left (\frac{1}{n}\sum_{k=1}^{n}  P\tilde{\eta}_{\chi} \circ F_{-k} \right ) \cdot \frac{i}{\xi_{0}} \right \Vert_{\infty}\\
= & \left \Vert \frac{i}{\xi_{0}} \int_{X} P \tilde{\eta}_{\chi} \> d\mu \right \Vert_{\infty}\\
= & \> 0.
\end{split}
\]

\noindent Hence,
$$
\limsup_{n \to \infty} \left \Vert I + \frac{H(n)}{n}H^{-1}  \right \Vert_{op}<1.
$$


\noindent$(ii)$
\[
\begin{split}
\left [P,\frac{H(n)}{n}H^{-1}\right ]= & \left [P, I + (\frac{1}{n}\sum_{k=1}^{n}  P\tilde{\eta}_{\chi} \circ F_{-k}) \cdot \frac{i}{\xi_{0}}I \right ]\\
= & \left [P, \left (\frac{1}{n}\sum_{k=1}^{n}  P\tilde{\eta}_{\chi} \circ F_{-k} \right ) \cdot \frac{i}{\xi_{0}}I \right ]\\
= & \left (\frac{1}{n}\sum_{k=1}^{n}  P(P\tilde{\eta}_{\chi}) \circ F_{-k} \right ) \cdot \frac{i}{\xi_{0}}I.
\end{split}
\]
\noindent Since $\sup_{t > 0} \left \Vert \frac{\mathscr{L}_{Y}(\chi \circ \eta) \circ F_{t} - \mathscr{L}_{Y}(\chi \circ \eta)}{t}) \right \Vert_{\infty} < \infty$, $P(P\tilde{\eta}_{\chi})$ is bounded in $L_{\chi}$ and
$$
\left \Vert \left [P,\frac{H(n)}{n}H^{-1} \right ] f \right \Vert_{L_{\chi}} \leq \frac{\Vert P(P\tilde{\eta}_{\chi}) \Vert_{L_{\chi}}}{|\xi_{0}|} \cdot \Vert f \Vert_{L_{\chi}} \leq \frac{C_{2}}{|\xi_{0}|} \Vert f \Vert_{L_{\chi}}.
$$
\noindent Thus, $\left [P,\frac{H(n)}{n}H^{-1} \right ]$ extends to a bounded operator on $L_{\chi}$ with uniformly bounded norm in $n$,
$$
\left \Vert \left [P,\frac{H(n)}{n}H^{-1} \right ] \right \Vert_{op} \leq \frac{C_{2}}{|\xi_{0}|}.
$$

\noindent $(iii)$ Since the operator $H$ is just multiplication by the constant $-i\xi_{0}$,

\begin{center} $\left [H(n),H\right ]H^{-1}=0$.

\end{center}

\noindent Thus, condition $(iii)$ is immediately satisfied.

%
Since conditions $(i)$, $(ii)$, and $(iii)$ of Theorem \ref{main} are satisfied on each $L_{\chi}$, we have shown that the operator $U_{\chi}$ has purely absolutely continuous spectrum on $L_{\chi}$. Thus, $W$ has purely absolutely continuous spectrum when restricted to the subspace $\bigoplus_{\chi \in \hat{G},\chi \circ \xi \not\equiv 1 }L_{\chi}$.

%
In addition, from the purity law in \cite{helson} extended to translations, the maximal spectral type is either purely Lebesgue, purely  singularly continuous, or purely discrete with respect to $\mu$ (the Haar measure). Since we know that the spectrum is absolutely continuous from above, we have proved the following theorem.

\begin{theorem}
 The operator $U_{\chi}$ has Lebesgue spectrum on $L_{\chi}$. Thus, $W$ has countable Lebesgue spectrum when restricted to the subspace $\bigoplus_{\chi \in \hat{G},\chi \circ \xi \not\equiv 1 }L_{\chi}$.
\end{theorem}

Similar results with less restrictive assumptions on $\eta$ have been derived in \cite{iwank} and \cite{tiedra}.

\subsection{Furstenberg transformations}

%
Let $\mu_{n}$ be the normalized Haar measure on $\mathbb{T}^{n} \simeq \mathbb{R}^{n}/\mathbb{Z}^{n}$  and $\mathscr{H}_{n}=L^{2}(\mathbb{T}^{n}, \mu_{n})$. Let $T_{d}:\mathbb{T}^{d} \to \mathbb{T}^{d}$, $d \geq 2$, be the uniquely ergodic  map \cite{fursten}
\[
T_{d}(x_{1}, x_{2}, ..., x_{d})
\]
\[
=(x_{1}+y, x_{2}+b_{2,1}x_{1}+h_{1}(x_{1}), ..., x_{d}+b_{d,1}x_{1}+ \cdot \cdot \cdot + b_{d,d-1}x_{d-1}+h_{d-1}(x_{1},x_{2},...,x_{d-1}))
\]
\[
\mod  \mathbb{Z}^{d}
\]
\noindent for $y \in \mathbb{R}\setminus \mathbb{Q}$, $b_{j,k} \in \mathbb{Z}$, $b_{l,l-1} \neq  0$, and $ l \in \{2,...,d\}$. (For $n=2$, we get the skew product in 4.1). Let each $h_{j}:\mathbb{T}^{j} \to \mathbb{R}$ satisfy a uniform Lipschitz condition in $x_{j}$ and be in $C^{2}(\mathbb{T}^{j})$. What follows is very similar to the case of the skew products over translations. We begin by considering the operator
\[
W_{d}: \mathscr{H}_{d} \to \mathscr{H}_{d}.
\]
\noindent The space $\mathscr{H}_{d}$ can be decomposed into
\begin{center} $\mathscr{H}_{d} = \mathscr{H}_{1} \bigoplus _{j \in \{2,...d \},\> k \in \mathbb{Z} \setminus \{0 \}} \mathscr{H}_{j,k}$ \end{center}
for $\mathscr{H}_{j,k}=\overline{Span\left \{\eta \bigotimes \chi_{k}|\eta \in \mathscr{H}_{j-1}\right \}}$ and $\chi_{k}(x_{j}) = e^{2 \pi i k x_{j}} \in \hat{\mathbb{T}}$.

The restriction of $W_{d}$, $W_{d}|_{\mathscr{H}_{j,k}}$ is unitarily equivalent to the operator
\[
U_{j,k} \eta = e^{2\pi i k \phi_{j}}W_{j-1}\eta
\]
\noindent for
\[
\eta \in \mathscr{H}_{j-1}
\]
\noindent and
\[
\phi_{j}(x_{1},x_{2},...,x_{j-1})=b_{j,1}x_{1}+ \cdot \cdot \cdot + b_{j,j-1}x_{j-1} + h_{j-1}(x_{1},x_{2},...,x_{j-1}).
\]
\noindent We will choose to take the commutator with $P_{j-1} = -i \partial_{j-1}$, the  essentially self-adjoint  \cite{amreinnelson} generator of the translation group $\{V_{t,j-1}\}_{t  \in \mathbb{R}}$ in $\mathscr{H}_{j-1}$. The following formulas hold on $D = C^{\infty}(\mathbb{T}^{j-1})$.
\[
\begin{split}
\left [P_{j-1},U_{j,k} \right ]= &\left [P_{j-1}, e^{2 \pi i k \phi_{j}}I \right ]W_{j-1}\\
= & -i\partial_{j-1}(e^{2 \pi i k \phi_{j}}) W_{j-1}\\
= & (2 \pi k b_{j,j-1} + 2 \pi k \partial_{j-1} h_{j-1})e^{2 \pi i k \phi_{j}}W_{j-1}.
\end{split}
\]
\noindent So,
\[
\left [P_{j-1}, U_{j,k}\right ] =(2 \pi kb_{j,j-1}+ 2 \pi k \partial_{j-1}h_{j-1})U_{j,k}= G U_{j,k}.
\]
\noindent
Thus,
\[ U_{j,k}^{-n}\left [P_{j,k}, U_{j,k}^{n} \right ] = \boxed{\sum^{n}_{l=1} U^{-l}_{j,k}GU^{l}_{j,k} =\sum^{n}_{l=1} G \circ T_{j-1}^{-l}=H(n).}
\]

\noindent From unique ergodicity we get the following convergence
\[
\lim_{n \to \infty} \frac{H(n)}{n}u =2 \pi k b_{j,j-1} + 2\pi k \int_{\mathbb{T}^{j-1}} \partial_{j-1}h_{j-1} \> d\mu =\boxed{2 \pi k b_{j,j-1}u=Hu}
\]

\noindent uniformly in $n$ for $u \in D=\mathscr{H}_{j-1}$.

%
Since
\[
\left \Vert \; \langle U_{j,k}^{n}f,f \rangle _{\mathscr{H}_{j-1}}\; \right \Vert_{\ell^{2}(\mathbb{Z})} \; = \; \left \Vert  \frac{1}{\sigma}\int_{0}^{\sigma} \langle e^{sP_{j-1}}U_{j,k}^{n}f, e^{sP_{j-1}}f \rangle_{\mathscr{H}_{j-1}} ds\> \right \Vert_{\ell^{2}(\mathbb{Z})},
\]

\noindent the preliminary assumptions for Theorem \ref{main} are satisfied for $B_{1}=B_{2}=I$.  Now we proceed by showing that the conditions of Theorem \ref{main} hold.


\noindent $(i)$ It is unnecessary to consider coboundaries since both $H=2 \pi k b_{j,j-1}I$ and $H^{-1}=\frac{1}{2 \pi k b_{j,j-1}}I$ are constants. Instead we take any $f \in \mathscr{H}_{j-1}$.
\[
\begin{split}
\frac{H(n)}{n}H^{-1}f = &
\left (\frac{1}{n}\left (\sum^{n}_{l=1} \left (2 \pi kb_{j,j-1}+ 2 \pi k \partial_{j-1}h_{j-1} \right ) \circ T_{j-1}^{-l} \right ) \right) \cdot \frac{1}{2 \pi kb_{j,j-1}}f\\
= & f + \left (\frac{1}{n}\sum^{n}_{l=1} \left (2 \pi k \partial_{j-1}h_{j-1} \right ) \circ T_{j-1}^{-l} \right )\frac{1}{2 \pi kb_{j,j-1}}f.
\end{split}
\]
\noindent Hence,
\[
\left \Vert \frac{H(n)}{n}H^{-1}f \right \Vert_{\mathscr{H}_{j-1}} \leq \left (1+\frac{\Vert 2 \pi k \partial_{j-1}h_{j-1} \Vert_{\mathscr{H}_{j-1}}}{|2 \pi kb_{j,j-1}|} \right ) \Vert f \Vert_{\mathscr{H}_{j-1}}.
\]
\noindent Since $h_{j-1}$ satisfies a uniform Lipschitz condition in $x_{j-1}$,
\[
\left \Vert \partial_{j-1}h_{j-1} \right \Vert_{\mathscr{H}_{j-1}}
\leq C_{1}.
\]
\noindent So  $\frac{H(n)}{n}H^{-1}$ extends to a bounded operator on $\mathscr{H}_{j,k}$ with uniformly bounded norm in $n$,
\[
\left \Vert \frac{H(n)}{n}H^{-1} \right  \Vert_{op} \leq 1+\frac{| 2 \pi k| \left \Vert \partial_{j-1}h_{j-1} \right \Vert_{\mathscr{H}_{j-1}}}{|2 \pi kb_{j,j-1}|} \leq \frac{C_{1}}{|b_{j,j-1}|}.
\]
\noindent Also,
\[
\begin{split}
& \left \Vert \left (I - \frac{H(n)}{n}H^{-1} \right ) f \right \Vert_{\mathscr{H}_{j-1}}\\
= & \left \Vert \left (1- \left (1+\left (\frac{1}{n}\sum^{n}_{l=1} (2 \pi k \partial_{j-1}h_{j-1}\right ) \circ T_{j-1}^{-l}\right )\frac{1}{2 \pi kb_{j,j-1}}\right )f \right \Vert_{\mathscr{H}_{j-1}}\\
= & \left \Vert \left (\frac{1}{n}\sum^{n}_{l=1} \left (2 \pi k \partial_{j-1}h_{j-1} \right ) \circ T_{j-1}^{-l}\right )\frac{1}{2 \pi kb_{j,j-1}}f \right \Vert_{\mathscr{H}_{j-1}}\\
\leq & \left \Vert \left (\frac{1}{n}\sum^{n}_{l=1} \left (2 \pi k \partial_{j-1}h_{j-1} \right ) \circ T_{j-1}^{-l} \right )\frac{1}{2 \pi kb_{j,j-1}} \right \Vert_{\infty} \cdot \Vert f \Vert_{\mathscr{H}_{j-1}}.
\end{split}
\]
\noindent As a result of unique ergodicity, the following converges uniformly,
\[
\lim_{n \to \infty} \left (\frac{1}{n}\sum^{n}_{l=1} (2 \pi k \partial_{j-1}h_{j-1}) \circ T_{j-1}^{-l} \right )\frac{1}{2 \pi kb_{j,j-1}} =\frac{1}{b_{j,j-1}} \int_{\mathbb{T}^{j-1}} \partial_{j-1}h_{j-1} \> d\mu=0,
\]
 \noindent and thus,
\[
\begin{split}
\limsup_{n \to \infty}\left  \Vert I - \frac{H(n)}{n}H^{-1}  \right \Vert_{op}\leq & \limsup_{n \to \infty} \left \Vert \left (\frac{1}{n}\sum^{n}_{l=1} (2 \pi k \partial_{j-1}h_{j-1}) \circ T_{j-1}^{-l} \right )\frac{1}{2 \pi kb_{j,j-1}} \right \Vert_{\infty}\\
= &\left \Vert\frac{1}{b_{j,j-1}} \int_{\mathbb{T}^{j-1}} \partial_{j-1}h_{j-1} \> d\mu \> \right \Vert_{\infty}=0
\end{split}
\]
\noindent Hence,
$$
\limsup_{n \to \infty} \left \Vert I + \frac{H(n)}{n}H^{-1}  \right \Vert_{op}<1.
$$
\noindent $(ii)$
\[
\begin{split}
\left [P_{j-1},\frac{H(n)}{n}H^{-1} \right ]= & \left [P_{j-1}, I + \left (\frac{1}{n}\sum_{k=1}^{n}  (2 \pi k \partial_{j-1}h_{j-1}) \circ T_{j-1}^{-l} \right ) \cdot \frac{1}{2 \pi kb_{j,j-1}}I \right ]\\
= & \left [P_{j-1}, \left (\frac{1}{n}\sum_{k=1}^{n}  (2 \pi k \partial_{j-1}h_{j-1}) \circ T_{j-1}^{-l} \right ) \cdot \frac{1}{2 \pi kb_{j,j-1}}I \right ]\\
= & \left (\frac{1}{n}\sum_{k=1}^{n}  (2 \pi k \partial_{j-1}(\partial_{j-1}h_{j-1})) \circ T_{j-1}^{-l} \right ) \cdot \frac{1}{2 \pi kb_{j,j-1}}I.
\end{split}
\]
\noindent Since $h_{j-1} \in C^{2}(\mathbb{T}^{j-1})$, $\partial_{j-1}(\partial_{j-1} h_{j-1})$ is bounded in $\mathscr{H}_{j-1}$,
\[
\begin{split}
\Vert [P,\frac{H(n)}{n}H^{-1}] f \Vert_{\mathscr{H}_{j-1}}
\leq & \frac{|2\pi k|\Vert \partial_{j-1}(\partial_{j-1} h_{j-1})\Vert_{\mathscr{H}_{j-1}}}{|2 \pi kb_{j,j-1}|} \cdot \Vert f \Vert_{\mathscr{H}_{j-1}}\\
\leq &  \frac{C_{2}}{|b_{j,j-1}|} \Vert f \Vert_{\mathscr{H}_{j-1}}.
\end{split}
\]
\noindent Thus, $[P,\frac{H(n)}{n}H^{-1}]$ extends to a bounded operator on $\mathscr{H}_{j-1}$ with uniformly bounded norm in $n$,
\[
\left \Vert [P,\frac{H(n)}{n}H^{-1}] \right \Vert_{op} \leq \frac{C_{2}}{|b_{j,j-1}|}.
\]


\noindent $(iii)$ Since the operator $H$ is just multiplication by $2 \pi k b_{j,j-1}$,
%
\[
\left [H(n),H \right ]H^{-1}=0.
\]

 Thus, condition $(iii)$ is immediately satisfied.

%
%
 Since conditions $(i)$, $(ii)$, and $(iii)$ of Theorem \ref{main} are satisfied, the operator $U_{j,k}$ has purely absolutely continuous spectrum on each $\mathscr{H}_{j,k}$. Thus, $W_{d}$ has purely absolutely continuous spectrum on the orthocomplement of  $\mathscr{H}_{1}$. Applying the the purity law in \cite{helson}, we derive, with slightly stronger regularity assumptions, a similar result to the ones found in \cite{iwank} and  \cite{tiedra}.

\begin{theorem}
$W_{d}$ has countable Lebesgue spectrum on the orthocomplement of  $\mathscr{H}_{1}$.
\end{theorem}

\section*{Acknowledgments}
\noindent The author is extremely grateful to her advisor, Professor Giovanni Forni, for his patience, guidance and support. The author would also like to thank the anonymous referee for various useful comments and suggestions.


Much of this work was completed at the University of Maryland, College Park, and was partially supported by Giovanni Forni's NSF Grant
DMS 1201534.


\end{document}